\documentclass[11pt,reqno,a4paper]{amsart}

\usepackage[utf8]{inputenc} 
\usepackage[T1]{fontenc}    
\usepackage{lmodern}
\usepackage[hidelinks]{hyperref}       
\usepackage{cleveref}
\usepackage{url}            
\usepackage{booktabs}       
\usepackage{amsthm,amsfonts, amsmath, amssymb, amsaddr, dsfont}       
\usepackage{nicefrac}       
\usepackage{microtype}      
\usepackage{natbib}
\usepackage{doi}
\usepackage{fancyhdr, graphicx, float, geometry, subcaption}
\usepackage{enumerate}
\usepackage{tikz, pgfplots}
\usepackage{ulem, xcolor}
\usepackage{epstopdf}
\usepackage{mathtools}

\pgfmathdeclarefunction{gauss}{2}{%
	\pgfmathparse{1/(#2*sqrt(2*pi))*exp(-((x-#1)^2)/(2*#2^2))}%
}
    \newcommand{\grd}{\nabla}
    \newcommand{\dt}{\partial_t}
\renewcommand{\d}{\,\text{d}}
    \newcommand{\rd}{\,\text{d}}
    \newcommand{\J}{\mathcal{J}}
    
    \newcommand{\K}{\mathcal{K}}
    \newcommand{\X}{\mathcal{X}}
    \newcommand{\D}{\mathcal{D}}
    \newcommand{\A}{\mathcal{A}}
    \newcommand{\B}{\mathcal{B}}
\renewcommand{\L}{\mathcal{L}}
    \newcommand{\C}{\mathcal{C}}
    \newcommand{\<}{\langle}
\renewcommand{\>}{\rangle}
    \newcommand{\rr}{\mathbb{R}}
    
    \newcommand{\eps}{\varepsilon}
\DeclareMathOperator*{\ess}{ess \sup}

\newcommand{\bigcupwithdot}{%
	\vcenter{\hbox{\ooalign{%
				$\bigcup$\cr
				\hidewidth$\cdot$\hidewidth\cr
	}}}%
}
\newcommand{\bigcupdot}{\mathop{\mathpalette\bigcupwithdot\relax}}

    \newtheorem{theorem}{Theorem}[section]
    \newtheorem{lemma}[theorem]{Lemma}
  \newtheorem*{design}{Design (D)}
    \newtheorem{corollary}{Corollary}[section]
    \newtheorem{proposition}{Proposition}[section]

\theoremstyle{definition}
    \newtheorem{definition}[theorem]{Definition}

\theoremstyle{remark}
    \newtheorem{remark}[theorem]{Remark}
    \newtheorem{assumption}{Assumption}

\makeatletter
\def\namedlabel#1#2{\begingroup
	\def\@currentlabel{#2}%
	\phantomsection\label{#1}\endgroup
}
\makeatother

\title[Reconstructing the kinetic chemotaxis kernel]{Reconstructing the kinetic chemotaxis kernel using macroscopic data: well-posedness and ill-posedness}

\author{Kathrin Hellmuth$^{a}$, Christian Klingenberg$^{a}$, Qin Li$^{b}$
	\and Min Tang$^{c}$
}
\address{$^{a}$Department of Mathematics, University of W\"urzburg, Germany\\
$^{b}$Department of Mathematics, University of Wisconsin-Madison, USA\\
$^{c}$School of Mathematics and Institute of Natural Sciences, MOE-LSC, Shanghai Jiaotong University, China}

\email{kathrin.hellmuth@uni-wuerzburg.de}
\email{klingen@mathematik.uni-wuerzburg.de}
\email{qinli@math.wisc.edu}
\email{tangmin@sjtu.edu.cn}

\begin{document}
\maketitle
 \begin{abstract}
Bacterial motion is steered by external stimuli (chemotaxis), and the motion described on the mesoscopic scale is uniquely determined by a parameter $K$ that models velocity change response from the bacteria. This parameter is called chemotaxis kernel. In a practical setting, it is inferred by experimental data. We deploy a PDE-constrained optimization framework to perform this reconstruction using velocity-averaged, localized data taken in the interior of the domain. The problem can be well-posed or ill-posed depending on the data preparation and the experimental setup. In particular, we propose one specific design that guarantees numerical reconstructability and local convergence. This design is adapted to the discretization of $K$ in space and decouples the reconstruction of local values of $K$ into smaller cell problems, opening up parallelization opportunities. Numerical evidences support the theoretical findings.
 \end{abstract}
\begin{quote}
	\noindent 
	{\small {\bf Keywords:} mathematical biology, kinetic chemotaxis model, parameter reconstruction,  macroscopic data, PDE-constrained optimization, well- and ill-posedness, inverse problem}
\end{quote}



\vspace{6pt} 





	\section{Introduction}\label{sec:introduction}
Kinetic chemotaxis equation is one of the classical equations that describes the collective behavior of bacteria motion. Presented on the phase space, the equation describes the ``run-and-tumble" bacteria motion \cite{AltChemotaxis,ErbanOthmerKineticChemotaxis,OthmerAltDispersalModels,OthmerHillen2Chemo}
\begin{align}
	\dt f + v\cdot \grd_x f &=\K(f):= \int_V K(x,v,v')f(x,t,v') - K(x,v',v) f(x,t,v)\rd v',\label{eq:chemotaxis}\\
	f(t=0,x,v) &= \phi(x,v)\,.  \label{eq:chemotaxis_init}
\end{align}
The solution $f(t,x,v)$ represents the density of bacteria at any given time $t$ for any location $x$ moving with velocity $v{\in V}$. The two terms describe different aspects of the motion. The $v\cdot\grd_xf$ term characterizes the ``run"-part: bacteria move in a straight line with velocity $v$, and the terms on the right characterize the ``tumble"-part: bacteria change from having velocity $v'$ to $v$ using the transitional rate $K(x,v,v')\geq 0$. This transition rate thus is termed the tumbling kernel. Initial data is given at $t=0$ and is denoted by $\phi(x,v)$. The equation contains phase-space information, and thus compared to the macroscopic models, such as the Keller Segel model, it offers more details and has the greater potential to capture the fine motion of the bacteria. Indeed, it is observed that the dynamics predicted by the model is in high agreement with real measurements, see~\cite{berg1993random,10.1371/journal.pcbi.1004843,10.1371/journal.pcbi.1000890, doi:10.1073/pnas.1101996108}.

It is noteworthy that these comparisons are conducted in the forward-simulation setting. Guesses are made about parameters, and simulations are run to be compared with experimental measurements. To fully reveal the bacteria's motion and its interaction with the environment, inverse perspectives have to be taken. This is to take measurements to infer $K$. The data can be collected at the individual level or the population level: biophysicists can use a high-resolution camera and trace each single bacterium for a long time to obtain single particle trajectory information, or take photos and record the density changes on a cell cultural dish. Such data should be used to unveil the true interaction between particles~\cite{doi:10.1073/pnas.1812570116}.

In this article, we frame this problem into a finite dimensional PDE-constrained optimization and study the unique and stable reconstructability of the kernel. In particular, we study different types of initial condition and measurement schemes and show that different experimental setups provide different stability of the reconstruction.

As more physics models derived from first-principles get deployed in applications, kinetic models are becoming more important in various scientific domains, see modeling of neutrons~\cite{davison1958neutron}, photons or electrons~\cite{RybickeLightman1986RadProc} and rarefied gas~\cite{cercignani2012boltzmann}. The applications on biological and social science have also been put forward in~\cite{OthmerAltDispersalModels} for cell motion, in~\cite{TaylorKingLoonChapman2014BirdMotion} for animal (birds) migration or in~\cite{albi2023datadriven,CarrilloToscanietal2009FlockingKineticEqOpinion,chu2022inference,MotschTadmore2014OpinionFormation,Toscani2006KineticOpinionModels} for opinion formation. In most of these models, parameters are included to characterize the interactions among agents or those between agents and the media. It is typical that these interactions cannot be measured directly, and it prompts the use of inverse solvers.

The most prominent application of inverse problem within the domain of kinetic systems is the optical tomography emerged from medical imaging, where non-intrusive {boundary data} is deployed to map out the optical properties of the interior. Mathematically a technique called the singular decomposition is deployed to conduct the inversion~\cite{ Bal_angularAveragedMeasurementsInstability, ChoulliStefanov_1996_SingDecomp, LaiLiUhlmann_2019_StatRTEInvProb,LiSun_2020_SingularDecomp, stefanov2022inverse}, and these studies have their numerical counterparts in~\cite{ Arridge_2009,ChenLiLiu_2018_OnlineLearningOptTomographySGD, EggerSchlottbom_2013_SRTEParamIdentification,Prieto_2017_RTEInvProblem, Ren_ReviewNumericsTransportBasedImag}, just to mention a few references.

Back to our current model, we notice that tracing the trajectory of every single bacterium is much more difficult than measuring the evolution of the macroscopic density~\cite{measureVelocityDependentf,BacteriaTracking_2010}, so we are tasked to unveil the interaction between bacteria and the environment using the density measurement. A series of new results by biophysicists~\cite{KineticParametersFromMacroMeasurements2,KineticParametersFromMacroMeasurements1} studies this experimental setting for a similar kinetic model and exhibits significance for practitioners. Compared with classical inverse problem originated from optical tomography, we encounter some new mathematical  challenges. In particular, in our setup, our measurements are taken in the interior of the domain instead of on the boundary, and 
interior data is richer than boundary measurements. Meanwhile, our data is velocity independent, as compared to that in optical tomography that contains velocity information, so we also lose some richness in data.

In~\cite{HKLT2022singulardecomposition} the authors examined the theoretical aspect of this reconstruction problem. It was shown that trading off the microscopic information for the interior data still gives us sufficient information to recover the transition kernel, but the experiments need to be carefully crafted. In this theoretical work we assumed that the transition kernel is an unknown function, and thus an infinitely dimensional object, and the available data is the full map (from initial condition to density for all time and space), and thus an infinite dimensional object as well. This infinite-to-infinite setup is hard to be implemented in a practical setting, rendering the theoretical results only a guidance for direct use. The current paper can be seen as the practical counterpart of~\cite{HKLT2022singulardecomposition}. In particular, our goal is to study the same question on the discrete level: when measurement data are finite in size, and the to-be-reconstructed transition kernel is also represented by a finite dimensional vector, can one still successfully recover the unknowns?

It turns out that the numerical issue is significantly more convoluted. In particular, when the dimension of $K$, the transition kernel, changes from infinite to finite, the amount of data needed to recover this parameter is expected to be reduced. The way of the reduction, however, is not clear. We will present below two different scenarios to argue:
\begin{itemize}
	\item when data is prepared well, a stable reconstruction is expected;
	\item when the data ``degenerates," it loses information, and the reconstruction does not hold.
\end{itemize}
Such coexistence of well-posedness and ill-posedness are presented respectively in two subsections of Section~\ref{sec:theory}. Then in Section~\ref{sec:NumExp} we present the numerical evidence to showcase the theoretical prediction.

It should be noted that it is well within anticipation that different data preparation gives different conditioning for parameter reconstruction. This further prompts the study of experimental design. In the context of reconstructing the transition kernel in the chemotaxis equation, in Section~\ref{sec:ExpDesign} we will design a particular experimental setup that guarantees a unique reconstruction. This verifies existence of the situation of data being well-prepared.

We should further mention that reconstructing parameters for bacterial motion using the inverse perspective is not entirely new. Until recently, existing literature followed two different approaches: the first involves the utilization of statistical information at the individual level to extrapolate the microscopic transition kernel~\cite{Pohl_2017_StatisticalTumbling, Seyrich_2018_StatisticalTumbling}, whereas the second entails employing density data at a macroscopic scale to reconstruct certain parameters associated with a macroscopic model through an optimization framework ~\cite{ChemotaxisRecoverDGamma, giometto2015generalized, ChemotacticSensitivityVariesALot, RecoverChemotacticSensivity}. To our knowledge, these available studies focus on a preset low-dimensional set of unknowns. The idea to infer parameters of kinetic descriptions from macroscopic type data emerged more recently~\cite{HKLT2022singulardecomposition, KineticParametersFromMacroMeasurements2, KineticParametersFromMacroMeasurements1}. The viewpoint we take in the current article significantly differs from those in the existing literature: Similar as was done in \cite{EggerPietschmannSchlottbohm_2015_UniqueChemotacticSensitivity, FisterMccarthy_IdentifyXiKS_2008} for a macroscopic model, we also recover the discretized version of the kinetic parameter. This brings more flexibility in application, at the cost of potentially high dimension of the unknown parameter. In contrast to existing results,
our focus lies on the study of identifiability of the parameter in the proposed optimization setting, and thus its  well- and ill-conditioning. Noise would introduce an additional layer of parameter uncertainty that we specifically seek to exclude from this stage of analysis. Numerical examples are thus presented in a noise-free and non regularized manner. This allows investigation of structural identifiability as well as suitability of specific experimental set ups to generate informative data for reconstruction in the sense of practical identifiability. 

\section{Framing a PDE-constrained optimization problem}\label{sec:Setting}
The problem is framed as a PDE-constrained optimization, which is to reconstruct $K$ that fits data as much as possible, conditioned on the fact that the kinetic chemotaxis model is satisfied.

We reduce the dimension of the original kinetic chemotaxis model \eqref{eq:chemotaxis}--\eqref{eq:chemotaxis_init} for $t>0$ from $(x,v)\in\mathbb{R}^3\times\mathbb{S}^2$ to $(x,v)\in\rr^1\times\{\pm 1\}$~\cite{giometto2015generalized, doi:10.1073/pnas.1101996108, 10.1371/journal.pcbi.1000890}, i.e. the bacteria either moves to the left or to the right, and $x$ is 1D in space. This simple setting reflects how experiments are conducted in the labs: bacteria are cultured in a tube, and the motion is one-dimensional. More details will be discussed in the subsequent part.

In a numerical setting, we first represent $K$ as a finite dimensional parameter{. After prescribing a  partition of the domain $\mathbb{R}^1=\bigcup_r I_r$ into intervals $I_r = [a_{r-1},a_r)$, for $r=2,...,R-1$, and $I_1 = (-\infty, a_1)$, $I_R= [a_{R-1},\infty)$, the function $K(x,v,v')$ is approximated within the cell $I_r$ by the value $K_r(v,v')$, constant in space:}
\begin{equation}\label{eq:Kstepfunction}
	K(x,v,v') = \sum_{r=1}^R
	K_r(v,v')\mathds{1}_{I_r}(x)\,,
\end{equation}
{where $\mathds{1}_I$ denotes the characteristic function of a subset $I\subset \rr^1$, i.e. $\mathds{1}_I(x) = 1$ if $x\in I$ and $0$ otherwise.}
{For $v=v'$, we set $K_r(v,v) = 0$ since these tumbling event cannot be distinguished from running a straight line and do not effect the motion \eqref{eq:chemotaxis}.} Since $V=\{\pm 1\}$, {then} there are only two parameters: $K_r(1,-1)$ and $K_r(-1,1)$ for each cell, so in total there are $2 R$ free values to represent $K$. Throughout the paper we abuse the notation and denote $K\in\mathbb{R}^{2R}$ as the unknown vector to be reconstructed, and denote:
\begin{equation}\label{eq:K_vec}
	K_r=[K_{r,1},K_{r,2}]\,\quad\text{with}\quad K_{r,i} = K_r(v_i,v_i')\,\quad \text{and} \quad (v_i,v_i') = \left((-1)^{i+1}, (-1)^i\right)\,
\end{equation}
for $i=1,2$.
The dataset is also finite in size. In particular, we mathematically represent {each} measurement as a reading of the bacteria density using a test function $\mu_l \in  {L^1}(\rr)$ for some $l$, so the measurement is:
\begin{equation}\label{eq:M}
	M_l(K) = \int_\rr\int_V f_K(x,T,v) \d v \ \mu_l(x)\rd x,\qquad l=1,...,L\,,
\end{equation}
where $f_K$ denotes the solution to \eqref{eq:chemotaxis} with kernel $K$. {Integration in velocity before testing with $\mu_l$ means that  only  the macroscopic density can be accessed.} {Even though integration amounts in a simple summation in our 1D setting with $V=\{\pm 1\}$, we stick to this notation for conciseness of the representation, and set $|V|:= \int_V\rd v$.}  In case $\mu_l$ is a characteristic function, this corresponds to the pixel reading of a photo. 

For simplicity of the presentation,
the ground-truth kernel $K_\star$ is assumed to be of form \eqref{eq:Kstepfunction} as well. Consideration of continuous in space ground truths would require additional approximation error estimates, as presented in \cite{jin2021error} for a diffusion coefficient reconstruction in elliptic and parabolic equations, which would go beyond the scope of this article. Then the true data is:
\begin{align}\label{eq:data}
	y_l = M_l(K_\star),\qquad   l=1,...,L\,.
\end{align}
Since $K$ is represented by a finite dimensional vector, we expect the amount of data needed {to be}  finite {as well}. Given the nonlinear nature of the {inverse} problem, it is unclear {whether} $L=2R$ leads to a unique reconstruction. One ought to dive in the intricate dependence on the form of $\{\mu_l\}_{l=1,...,L}$.

To conduct such inversion, we deploy a PDE-constrained optimization formulation. This is to minimize the square loss between the simulated data $M(K)$ and the data $y$:
\begin{equation}\label{eq:loss}
	\begin{aligned}
		\min_K\quad  \mathcal{C}(K) &=  \min\frac{1}{2L} \sum_{l=1}^L \left(M_l(K) -y_l\right)^2\\
		\quad \text{subject to}&\quad \text{\eqref{eq:chemotaxis}, and \eqref{eq:chemotaxis_init}}.
	\end{aligned}
\end{equation} 

Many algorithms can be deployed to solve this minimization problem, and we are particularly interested in the application of gradient-based solvers. The simple gradient descent (GD) method gives:
\begin{equation}\label{eq:GradientDescentStep}
	K^{(n+1)} =  K^{(n)} - \eta_n \nabla_K\C(K^{(n)})\,,
\end{equation}
with a suitable step size $\eta_n\in \rr_+$. It is a standard practice of calculus-of-variation to derive the partial differentiation against the {$(r,i)$}-th ($i=1,2$, $r=1,\cdots, R$) entry in the gradient $\nabla_K \mathcal{C}$:
\begin{equation}\label{eq:gradC}
	{ \frac{\partial \mathcal{C}}{\partial  K_{r,i}} = \int_0^T \int_{I_r} f(t,x,v_i')(g(t,x,v_i')-g(t,x,v_i))\rd x\rd t\,,}
\end{equation}
Detailed are placed in \Cref{sec:App:Opt} in  the supplementary materials. In the formulation, {$(v_i,v_i')$ is given in \eqref{eq:K_vec} and} $g$ is the adjoint state that solves the adjoint equation 
\begin{align}
	-\partial_t g -v\cdot \nabla g &=\tilde{\K}(g):=  \int_V K(x,v',v)(g(x,t,v')-g(x,t,v))\d v', \label{eq:adjoint}\\
	g(x,t=T,v) &=-\frac{1}{L}\sum_{l=1}^{L}\mu_l (x) \left(M_l(K) - y_l\right) \,.\label{eq:g_final}
\end{align}

The convergence of GD in~\eqref{eq:GradientDescentStep} is guaranteed for a suitable step size if the objective function is convex. Denoting $H_K\C$ the Hessian function of the loss function, we need $H_K\C>0$ at least in a small neighborhood around $K_\star$. In~\cite{WrightRecht_2022_OptimizationforDataAnalysis}, a constant step size $\eta_n = \eta =\frac{2\lambda_{\min}}{\lambda_{\max}^2}$ is recommended with $\lambda_{\min}, \lambda_{\max}$ denoting the smallest and largest eigenvalues of $H_K\C(K_\star)$. More sophisticated methods include line search for the step size or higher order methods are also possible, see e.g.~\cite{Ren_ReviewNumericsTransportBasedImag,WrightRecht_2022_OptimizationforDataAnalysis}.

To properly set up the problem, we make some general assumptions and fix some notations.
\begin{assumption}\label{ass:all}
	We make assumptions to ensure the well-posedness of the forward problem in a feasible set, in particular:
	\begin{itemize}
		\item We will work locally in $K$, so we assume in a neighbourhood $\mathcal{U}_{K_\star}$ of $K_\star$, there is a constant $C_K$ {  that uniformly bounds all $\hat K\in \mathcal{U}_{K_\star}$:}
		\begin{equation}\label{eq:Kbounded}
			0<\|\hat K\|_{\infty}\leq C_K\,{ \quad \text{ for all } \hat K \in \mathcal U_{K_\star}}.
		\end{equation}
		\item Assume {that} the initial data $\phi$ {is} in the  space {${L^{\infty}_{+}(\rr\times V)}\cap L^1(\rr;L^\infty(V))$} 
		of non negative functions with essential bound $$\|\phi\|_{L^\infty(\rr\times V)}{,\|\phi\|_{ L^1(\rr;L^\infty( V))}\leq} C_\phi\,.$$
		\item The test functions $\{\mu_l\}_{l=1}^L$ are supposed to be selected from the space ${L^1} (\rr)$ with uniform  $L^1$ bound
		\begin{align*}
			\int_\rr|\mu_l|\rd x\leq C_\mu, \quad l=1,...,L\,.
		\end{align*}   
	\end{itemize}
\end{assumption}

These assumptions are satisfied in a realistic setting. They allow us to operate $f$ and $g$ in the right spaces. In particular, we can establish existence of mild solutions and upper bounds for both the forward and adjoint solution, see Lemma~\ref{lem:existencef} and~\ref{lem:existenceg} in the  supplementary materials~\Cref{sec:a_priori}.

\section{Well-posedness vs. ill-posedness} \label{sec:theory}
As many optimization algorithms are designed to produce minimizing sequences, we study well-posedness in the sense of Tikhonov.
\begin{definition}[Tikhonov well-posedness \cite{Tykhonov_WellPosed}]
	A minimization problem is Tikhonov well-posed, if a unique minimum point exists towards which every minimizing sequence converges.
\end{definition}
The well-posedness of the inversion heavily depends on the data preparation. If a suitable experimental setting is arranged, the optimization problem is expected to provide local well-posedness around the ground-truth parameter $K_\star$, so the classical GD can reconstruct the ground-truth. However, if data becomes degenerate, we also expect ill-conditioning and the GD will find it hard to converge to the global minimum. We spell out the two scenarios in the two theorems below.

\begin{theorem}\label{thm:well-posedness}
	Assume the Hessian matrix of the cost function is positive definite at $K_\star$ and let the remaining assumptions of Proposition \ref{prop:well-posedness} hold, then there exists a neighbourhood $U$ of $K_\star$, in which the optimization problem \eqref{eq:loss} is Tikhonov well-posed. In particular, the gradient descent algorithm \eqref{eq:GradientDescentStep} with initial value $K_0\in U$ converges.
\end{theorem}
This theorem provides the well-posedness of the problem. To be specific, it spells out the sufficient condition for GD to find the global minimizer $K_\star$. The condition of the Hessian being positive definite at $K_\star$ may seem strong. In Section~\ref{sec:ExpDesign}, we will carefully craft a setting{, with $L=2R$,} for which we can ensure this to hold.

On contrary to the previous well-posedness discussion, we also provide a negative result below on ill-conditioning.
\begin{theorem}\label{thm:illconditioning}
	Let $L=2R$ and let Assumption \ref{ass:all} hold {for all considered quantities}. 
	Consider a sequence $(\mu_{{1,m}})_m$ of test functions  for the first measurement $M_1(K)$ { that converges  to the test function $\mu_2$ of the second measurement $M_2(K)$ as $m\to \infty$, either
		\begin{enumerate}[(1)]
			\item $\mu_{{1,m}}\to \mu_2$ strongly in $L^1$, or 
			\item \label{itm:IllCond:weakconv} $\mu_{{1,m}}\rightharpoonup \mu_2$ weakly in $L^1$. {In this case, further assume that the measurement time $T$ be chosen small  such that $e^{T|V|C_K}<2$.}
	\end{enumerate}}
	Then, as $m\to \infty$, i.e. as the measurement test functions become close in one of the above senses, strong convexity of the loss function decays, and the convergence of the gradient descent algorithm \eqref{eq:GradientDescentStep} to $K_\star$ {cannot be} guaranteed.
\end{theorem}
{The two scenarios describe two different qualities convergence: 
	\begin{enumerate}[(1)]
		\item strong convergence in $L^1$ norm $\|\mu_{{1,m}}-\mu_2\|_{L^1}\to 0$, and 
		\item weak convergence in $L^1$, a distributional property that requires a test function $h\in L^\infty$ to observe convergence $\int \mu_{{1,m}}(x)h(x)\d x\to \int \mu_2(x)h(x)\d x$.
	\end{enumerate}
	The two different settings require different proof techniques, and these technical constraints are the main reason for (2) to require additional assumption on $T$.
}
{
	\begin{remark}
		In the case of pointwise measurement, the test function $\mu_i$ are Dirac delta measures. Since $\delta_{x_i}\notin L^1(\rr)$, the case is not covered in our setting. However, a small modification to the proof lets us handle the situation where $\mu_i$ is a mollification of Dirac. If the mollification parameters are independent of $i$, having $x_{{1,m}}\to x_2$ puts us back in the setting of the theorem, and the result still holds.
	\end{remark} 
}

The two theorems, to be proved in detail in Section~\ref{ssec:wellposed} and~\ref{ssec:IllCond} respectively, hold vast contrast to each other. The core difference between the two theorems is the data selection. The former guarantees the convexity of the objective function, and the latter shows degeneracy. The analysis comes down to evaluating the Hessian, a $2R\times 2R$ matrix:
\begin{equation}\label{eq:HessianForm}
	H_K\C (K) = \frac{1}{L}\sum_{l=1}^{L}\big( \grd_K M_l(K)\otimes \grd_K M_l(K) + (M_l(K)-y_l)H_K M_l(K)\big)\,.
\end{equation}
It is a well-known fact~\cite{PolyakShcherbakov_2017_Optimization} that a positive definite Hessian provides the strong convexity of the loss function, and is a sufficient criterion that permits the convergence {of the parameter reconstruction.} 
If $H_K\C(K_\star)$ is known to be positive and the Hessian matrix does not change much under small perturbation of $K$, then convexity of the cost function can be guaranteed in a small environment around $K_\star$. Such boundedness of perturbation in the Hessian is spelled out in Proposition \ref{prop:well-posedness}, and Theorem~\ref{thm:well-posedness} naturally follows. 

Theorem~\ref{thm:illconditioning} orients the opposite side. In particular, it examines the degeneracy when two data collection points get very close. The guiding principle for such degeneracy is that when two measurements can get too close, they offer no additional information. Mathematically, this amounts to rank deficiency of the Hessian~\eqref{eq:HessianForm}, prompting the collapse of convexity in the landscape of the objective function. 
{Proposition~\ref{prop:gradconv} and~\ref{prop:gradconv_singular} consider the two different notions of closeness, each of which entails their own strategy  to control the vanishing information gain from the first measurement.}

\subsection{Local well-posedness of the optimization problem} \label{ssec:wellposed}

Generally speaking, it would not be easy to characterize the {full} landscape of the {loss function}  
and thus hard to prescribe conditions for obtaining global convergence. However, suppose the data is prepared well enough so to  guarantee the positive definiteness for the Hessian $H_K\C(K_\star)$ evaluated at the ground-truth $K_\star$, then the following results provide that in a small neighborhood of this ground-truth, positive-definiteness persists. Therefore, GD that starts within this neighborhood, finds the global minimum to~\eqref{eq:loss}. This gives us local well-posedness.

\begin{proposition}\label{prop:well-posedness}
	Let Assumption \ref{ass:all} hold. Assume the Hessian $H_K\C(K_\star)$ is positive definite at $K_\star$, {and that there exists a neighborhood {$\mathcal{U}_{K_\star}$} of $K_\star$ in which the  Hessian {of the measurements} are uniformly bounded in Frobenius norm, i.e. for all  $l=1,...,L$ {and $K\in {\mathcal{U}_{K_\star}}$} one has $\|H_KM_l(K)(v,v')\|_F\leq C_{H_KM}$.}
	Then there exists a (bounded)  neighbourhood $U\subset{\mathcal{U}_{K_\star}}$ of $K_\star$ {in $L^\infty$ norm}, where  $H_K\C(K)$ is positive definite for all $K\in U$. Moreover, the minimal eigenvalues $\lambda_{\min}(H_K\C)$ satisfies
	\begin{equation}\label{eqn:Hessian_bound}
		|\lambda_{\min}(H_K\C(K_\star)) - \lambda_{\min}(H_K\C(K))| \leq \|K_\star-K\|_\infty C',
	\end{equation}
	where the constant $C'$ depends on the measurement time $T$, $R$, and the bounds $C_\mu$, $C_\phi$, $C_K$ in Assumption \ref{ass:all} and $C_{H_KM}${, but not on $K$}. As a consequence, the radius of $U$ {can be chosen as} ${\lambda_{\min}(H_K\C(K_\star))}/{C'}.$
\end{proposition}

The proposition is hardly surprising. Essentially it suggests the Hessian term is Lipschitz continuous with respect to its argument. This is expected if the solution to the equation is somewhat smooth. Such strategy will be spelled out in detail in the proof. Now Theorem~\ref{thm:well-posedness} is immediate.
\begin{proof}[Proof for Theorem~\ref{thm:well-posedness}]
	By Proposition \ref{prop:well-posedness}, there exists a neighbourhood $U$ of $K_\star$ in which the Hessian  is positive definite, $H_K\C(K)>0$ for all $K\in U$. Without loss of generality, we can assume that  $U$ is a convex set. By the strong convexity of $\C$ in $U$, the minimizer $K_\star\in U$ of $\C$ is unique and thus the finite dimension of the parameter space $K\in \rr^{2R}$ guarantees Tikhonov well-posedness of the optimization problem \eqref{eq:loss}~\cite[Prop.3.1]{FerrentinoBoniello_TykhonovWellPosed}. Convergence of GD follows from strong convexity of $\C$ in $U$.
\end{proof}

Now we give the proof for Proposition~\ref{prop:well-posedness}. It mostly relies on the matrix perturbation theory~\cite[Cor. 6.3.8]{horn_johnson_1985} and continuity of equation \eqref{eq:chemotaxis} with respect to the parameter $K$.
\begin{proof}[Proof for Proposition~\ref{prop:well-posedness}]
	According to the matrix perturbation theory, the minimal eigenvalue is continuous with respect to a perturbation to the matrix, we have
	\begin{align}\label{eq:well-posedness:eigenvaluedifference}
		&|\lambda_{\min}(H_K\C(K_\star)) - \lambda_{\min}(H_K\C(K))| \leq \|H_K\C(K_\star) - H_K\C(K)\|_F \nonumber\\
		&\leq \frac{1}{L}\sum_l\bigg( \|(\grd_KM_l\otimes \grd_KM_l)(K_\star) - (\grd_KM_l\otimes \grd_KM_l)(K) \|_F  \nonumber\\
		&\hspace{1.3cm}+ \|(M_l(K)-y_l)H_KM_l(K)\|_F\bigg)\\
		&\leq \frac{1}{L}\sum_l \bigg( \|\grd_KM_l(K_\star) - \grd_KM_l (K)\|_F\left(\|\grd_KM_l(K_\star) \|_F +\|\grd_KM_l(K) \|_F\right)\nonumber \\
		&\hspace{1.3cm}+ |M_l(K)-y_l|\|H_KM_l(K)\|_F\bigg)\nonumber
	\end{align}
	where we used the Hessian form~\eqref{eq:HessianForm}, triangle inequality and sub-multiplicativity for Frobenius norms. To obtain the bound~\eqref{eqn:Hessian_bound} now amounts to quantifying each term on the right hand side of~\eqref{eq:well-posedness:eigenvaluedifference} and bounding them by $\|K_\star-K\|_\infty$. This is respectively achieved in Lemmas \ref{lem:MeasDiffBound},  \ref{lem:GradBound} and   \ref{lem:GradDiffBound} that give controls to $M_l(K)-y_l$, $\|\grd_KM_l(K) \|_F$ and $\|\grd_KM_l(K_\star) - \grd_KM_l(K) \|_F $. Putting these results together, we have:

	\begin{align*}
		&|\lambda_{\min}(H_K\C(K_\star)) - \lambda_{\min}(H_K\C(K))| \leq \|H_K\C(K_\star) - H_K\C(K)\|_F\\
		&\leq  2\|K_\star - K\|_\infty C_\mu C_\phi e^{2C_K|V|T}{T}\Bigg[ 8R  C_\phi C_\mu e^{2|V|C_KT}  \left(|V|T^2 + \frac{1}{C_K}\left(\frac{e^{2C_K|V|T}-1}{2C_K|V|}-T\right)\right) \\
		&\hspace{4.6cm}+  |V|^2 C_{H_KM} \Bigg]\\
		&=:  \|K_\star - K\|_\infty C'.
	\end{align*}
	The positive definiteness in a small neighborhood of $K_\star$ now follows: {Given $\|K_\star-K\|_\infty < \lambda_{\min}(H_K\C(K_\star))/C'$, the triangle inequality shows 
		\[\lambda_{\min}(H_K\C(K)) \geq \lambda_{\min}(H_K\C(K_\star))- |\lambda_{\min}(H_K\C(K_\star))-\lambda_{\min}(H_K\C(K))|  > 0.
		\] }
\end{proof}

As can be seen from the proof, Proposition~\ref{prop:well-posedness} strongly relies on the boundedness of the terms in~\eqref{eq:well-posedness:eigenvaluedifference}. We present the estimates below.
\begin{lemma}\label{lem:MeasDiffBound}
	Let Assumptions \ref{ass:all} holds, then the measurement difference is upper bounded by:
	\begin{align*}
		|M_l(K)-y_l|\leq |V|C_\mu \|(f_{K_\star}-f_K)(T)\|_{L^{\infty}(\rr\times V)} \leq  \|K_\star-K\|_\infty 2|V|^2C_\mu C_\phi Te^{2C_K|V|T}.
	\end{align*}
\end{lemma}
\begin{proof} 
	Apply Lemma \ref{lem:existencef} to the difference equation for $\bar f := f_{K_\star}-f_K$ 
	\begin{align}\label{eq:fKdifference}
		\dt\bar f + v\cdot \nabla_x \bar f =  \K_K(\bar f) +\K_{(K_\star - K)}(f_{K_\star})
	\end{align}
	with initial condition $0$ and source {term} $h= \K_{(K_\star - K)}(f_{K_\star}) \in {L^1}((0,T);L^\infty(\rr\times V){\cap L^1(\rr;L^\infty(V))})$ by the regularity \eqref{eq:regulatityf} of $f_{K_\star}$. This leads to
	\begin{align}
		\ess_{v,x}|\bar f|(x,t,v) 
		\leq &\int_0^t e^{2|V|C_K(t-s)}{\|\K_{(K_\star - K)}(f_{K_\star})(s)\|_{L^\infty(\rr\times V) \cap L^1(\rr;L^\infty(V))}}
		\rd s \nonumber\\
		\leq &2|V|\|K_\star-K\|_\infty e^{2|V|C_Kt}C_\phi t, \label{eq:Well-posedness:fdifference}
	\end{align}
	where we used the estimate $\|f_{K_\star}(s)\|_{L^\infty(\rr\times V){\cap L^1(\rr;L^\infty(V))}} \leq e^{2|V|C_K s}{C_\phi} $  
	from Lemma \ref{lem:existencef} in the last step.
\end{proof}
To estimate the gradient $\grd_K M_l(K)$ and its difference, we first recall the form in~\eqref{eq:gradC} with $\mathcal{C}$ changed to $M_l$ here. Analogously, we can use the adjoint equation to explicitly represent the gradient:
\begin{lemma}\label{lem:GradientForm}
	Let Assumption~\ref{ass:all} hold. Denote by $f_K$ the {mild} solution of \eqref{eq:chemotaxis} and by $g_l\in C^0\left([0,T];L^\infty(V;{L^1}(\rr))\right)$ the {mild} solution of 
	\begin{align}\label{eq:adjointM}
		-\partial_t g_l -v\cdot \nabla g_l&=\tilde{\K}(g_l):=  \int_V K(x,v',v)(g_l(x,t,v')-g_l(x,t,v))\d v', \\ 
		g_l(t=T, x,v) &= -\mu_l(x)\,.\nonumber
	\end{align}
	Then
	\begin{equation}\label{eq:gradMAdjointFormula}
		{\frac{\partial M_l(K)}{\partial K_{r,i} } = \int_0^T\int_{I_r} f'(g_l'-g_l) \rd x\rd t\,,}
	\end{equation}
	where we used the abbreviated notation $h:= h(t,x,v_i)$ and $h':= h(t,x,v_i')$ for $h = f,g_l$, {with $(v_i,v_i')$ defined as in \eqref{eq:gradC}.} 
\end{lemma}
We omit explicitly writing down the $x,t$ dependence when it is not controversial. The proof for this lemma is the application of calculus-of-variation and will be omitted from here. We are now in the position to derive the  estimates of the gradient norms.
\begin{lemma}\label{lem:GradBound}
	Under Assumption \ref{ass:all}, the gradient is uniformly bounded 
	\begin{align*}
		\|\grd_KM_l(K) \|_F \leq  \sqrt{2R} 2C_\phi C_\mu e^{2C_K|V|T} T, \qquad \text{ for all } K\in \mathcal{U}_{{K_\star}}.
	\end{align*}
\end{lemma}

\begin{proof}
	The Frobenius norm is bounded by the entries 
	\[
	\|\grd M_l(K)\|_F\leq  \sqrt{2R} \max_{r,i} \left|\frac{{\partial} M_l(K)}{ {\partial} K_{r,i}}\right|.
	\]
	Representation
	\eqref{eq:gradMAdjointFormula} together with \eqref{eq:boundf} then gives the bound
	\begin{align}\label{eq:GradBounded}
		\left|\frac{{\partial} M_l}{{\partial} K_{r,i}}\right| \leq 2C_\phi\int_0^Te^{2|V|C_Kt}\max_v\left(\int_\rr\left|g_l\right|\rd x  \right)\rd t,
	\end{align}
	Application of lemma \ref{lem:existenceg} to $g=g_l$, with $h=0$ and $\psi = -\mu_l$, yields
	\begin{align}\label{eq:glessbounded}
		\max_v \int_\rr\left|g_l\right|\rd x \ (t) \leq \int_\rr|-\mu_l(x) |\rd x \ e^{2C_K|V|(T-t)} \leq C_\mu e^{2C_K|V|(T-t)},
	\end{align}
	which, when plugged into \eqref{eq:GradBounded}, gives
	\begin{align*}
		\left|\frac{\partial M_l}{\partial K_{r,i}}\right| \leq 2C_\phi C_\mu e^{2C_K|V|T} T \,.
	\end{align*} 
\end{proof}

\begin{lemma}\label{lem:GradDiffBound}
	In the setting of Theorem \ref{thm:well-posedness} and under Assumption~\ref{ass:all}, the gradient difference is uniformly bounded in $K\in \mathcal{U}_{{K_\star}}$ by
	\begin{align*}
		&\|\grd M_l(K_\star)-\grd M_l(K)\|_F \\
		&\leq  \sqrt{2R}\|K_\star-K\|_\infty 2C_\phi C_\mu e^{2C_K|V|T}\left(|V|T^2 + \frac{1}{C_K}\left(\frac{e^{2C_K|V|T}-1}{2C_K|V|}-T\right)\right)\,.
	\end{align*}
\end{lemma}
\begin{proof}
	Now consider the entries of $\grd M_l(K_\star)-\grd M_l(K)$ to show smallness of  $\|\grd M_l(K_\star)-\grd M_l(K)\|_F$.
	Rewrite, using lemma \ref{lem:GradientForm} and \eqref{eq:boundf},  
	\begin{align*}
		\left|\frac{\partial M_l(K_\star)}{\partial K_{r,i}} - \frac{\partial M_l(K)}{\partial K_{r,i}}\right| = &\left|\int_0^T\int_{I_r} f_{K_\star}({g'}_{l,K_\star}  - g_{l,K_\star} ) - f_{K}({g'}_{l,K}  -g_{l,K})\rd x\rd t
		\right|\\
		\leq &\int_0^T\|(f_{K_\star}  -f_K)(t)\|_{L^\infty(\rr\times V)}2\max_{v}\int_\rr |{g}_{l,K_\star}(t)|\rd x \rd t\\
		& + 2  C_\phi \int_0^T e^{2|V|C_K t} \max_v\int_\rr |(g_{l,K_\star} - g_{l,K})(t)|\rd x \rd t.
	\end{align*}
	
	The first summand can be bounded by  \eqref{eq:Well-posedness:fdifference} and \eqref{eq:glessbounded}.
	To estimate the second summand, apply Lemma \ref{lem:existenceg} to $\bar g :=g_{l,K_\star} - g_{l,K} $ with evolution equation
	\begin{align*}
		-\dt \bar g - v\cdot \grd_x \bar g &= \tilde{\K}_{K_\star}(\bar g)+ \tilde{\K}_{(K_\star-K)}(g_{l,K}),\\
		\bar g (t=T) &= 0,
	\end{align*}
	and $h= \tilde{\K}_{(K_\star-K)}(g_{l,K}) \in {L^1} ((0,T);L^\infty( V;L^1(\rr)))$ by the regularity \eqref{eq:regularityg} of $g_{l,K}\in  \\
	C^0\left((0,T);L^\infty(V;L^1(\rr))\right)$. This leads to
	\begin{align*}
		\max_v\int_\rr|\bar g|\rd x &\leq e^{2|V|C_K(T-t)}\int_0^{T-t}\max_v\|\tilde{\K}_{(K_\star-K)}(g_{l,K})(T-s,v)\|_{L^1(\rr)}\rd s\\
		&\leq  2|V|\|K_\star-K\|_\infty e^{2|V|C_K(T-t)}\int_0^{T-t}\max_v\|g_{l,K}(T-s,v)\|_{L^1(\rr)}\rd s\\
		&\leq
		\|K_\star-K\|_\infty \frac{C_\mu}{C_K}e^{2|V|C_K(T-t)}( e^{2C_K|V|(T-t) } -1) ,
	\end{align*}
	where we used \eqref{eq:glessbounded} in the last line.
	In summary, one obtains
	\begin{align*}
		&\left|\frac{\partial M_l(K_\star)}{\partial K_{r,i}} - \frac{{\partial} M_l(K)}{{\partial} K_{r,i}}\right|\\
		&\leq \|K_\star-K\|_\infty \bigg[\int_0^T 2|V|C_\phi te^{2C_K|V|t}\cdot 2C_\mu e^{2C_K|V|(T-t)} \rd t \\
		&\hspace{2cm}+ 2C_\phi \int_0^T e^{2|V|C_K t}\frac{C_\mu}{C_K}e^{2C_K|V|(T-t)}(e^{2C_K|V|(T-t)}-1)\rd t \bigg]\\
		&\leq \|K_\star-K\|_\infty 2C_\phi C_\mu e^{2C_K|V|T}\left(|V|T^2 + \frac{1}{C_K}\left(\frac{e^{2C_K|V|T}-1}{2C_K|V|}-T\right)\right).
	\end{align*}
\end{proof}
Together with the boundedness of the gradient \eqref{eq:GradBounded}, this shows that the first summands in  \eqref{eq:well-posedness:eigenvaluedifference} are Lipschitz  continuous in $K$ around $K_\star$ which concludes the proof of Proposition \ref{prop:well-posedness}.

\subsection{Ill-conditioning for close measurements} \label{ssec:IllCond}
While the positive Hessian at $K_\ast$ guarantees local convergence, such positive-definiteness will disappear when data are not prepared well. In particular, if $L=2R$, meaning the number of measurements equals the number of parameters to be recovered, and that two measurements, $M_1(K)$ and $M_2(K)$ are close, we will show that the Hessian degenerates. Then strong convexity is lost, and the convergence to $K_\star$ is no longer guaranteed.

We will study how the Hessian degenerates in the two scenarios in Theorem~\ref{thm:illconditioning}. This comes down to examining the two terms in~\eqref{eq:HessianForm}. Applying Lemma \ref{lem:MeasDiffBound}, we already see the second part in~\eqref{eq:HessianForm} is negligible {when $K$ is close to $K_\star$} and the rank structure of the Hessian is predominantly controlled by the first term. It is a summation of $L$ rank 1 matrices $\grd_K M_l(K)\otimes \grd_K M_l(K) $. When two measurements ($\mu_1$ and $\mu_2$) get close, we will argue that $\grd_K M_1(K)$ is almost parallel to $\grd_K M_2(K)$, making the Hessian {lacking at least }
one rank, and the strong convexity is lost. Mathematically, this means we need to show $\|\grd_K M_1(K)- \grd_K M_2(K) \|_{{F}} \approx 0$ when $\mu_1\approx\mu_2$.

Throughout the derivation, the following formula is important. Recalling~\eqref{eq:gradMAdjointFormula}, we have for every $r\in\{1,\cdots,R\}$ and $i\in\{1,2\}$
\begin{align}
	\frac{\partial M_1(K)}{\partial K_{r,i} } - \frac{\partial M_2(K)}{\partial K_{r,i} } &= \int_0^T\int_{I_r} f'((g_1-g_2)'-(g_1 -g_2))\rd x\rd t  \nonumber\\&=\int_0^T \int_{I_r} f'(\bar g'-\bar g) \rd x\rd t\,,\label{eq:gradMdifference}
\end{align}
where $\bar g := g_1-g_2$ solves \eqref{eq:adjoint} with final condition $\bar g (t=T,x,v) = \mu_2(x) -\mu_1(x)$. The two subsections below serve to quantify the smallness of~\eqref{eq:gradMdifference} in terms of the smallness of $\mu_1(x)-\mu_2(x)$.

\subsubsection{{Closeness in the strong sense}}
The following proposition states the loss of strong convexity as $\mu_2 - \mu_{{1,m}} \to 0$ in $L^1(\rr)$. In particular, the requirement of Proposition \ref{prop:well-posedness} that $H_K\C(K_\star)$ is positive definite is no longer satisfied, so local well-posedness of the optimization problem and thus the convergence of the algorithm can no longer be guaranteed.

\begin{proposition}\label{prop:gradconv}
	Let Assumption~\ref{ass:all} hold.
	Then, as $\mu_{{1,m}}\xrightarrow{m\to \infty} \mu_2$ in $L^1(\rr)$, {one eigenvalue of the Hessian $H_K\C(K_\star)$ vanishes.} 
\end{proposition}

This proposition immediately allows us to prove  {the assertion of } Theorem~\ref{thm:illconditioning} {for $L^1$ closeness}:
\begin{proof}[Proof of Theorem \ref{thm:illconditioning}]
	Propositions \ref{prop:gradconv} establishes {one }
	eigenvalue of  $H_K\C(K_\star)$ vanishes as $m\to \infty$. This {lack of }
	positive definiteness and thus strong convexity of $\C$ around $K_\star$ means that it {cannot be} 
	guaranteed that the minimizing sequences of $\C$ converge to $K_\star$.
\end{proof}

\begin{proof}[{Proof of \Cref{prop:gradconv}}]
	As argued above, {our goal is to} show  $\|\grd_K M_{{1,m}}(K)- \grd_K M_2(K) \|_{{F}} \to 0$ as $m\to \infty$. Recall~\eqref{eq:gradMdifference}, we need to show:
	\begin{align}
		\frac{\partial M_{{1,m}}(K)}{\partial K_{r,i} } - \frac{\partial M_2(K)}{\partial K_{r,i} }\xrightarrow{m\to\infty}0\quad \forall (r,i)\in\{1,\cdots,R\}\times\{1,2\}\,.
	\end{align}
	where $\bar g_{{m}} := g_{{1,m}}-g_2$ solves {\eqref{eq:adjointM}} with final condition $\bar g_{{m}} (t=T,x,v) = \mu_2(x) -\mu_{{1,m}}(x)$. 
	Application of  Lemma \ref{lem:existenceg} gives
	\begin{align}\label{eq:gbarnorm}
		\|\bar g_{{m}}(t)\|_{L^\infty(V;L^1(\rr))} \leq 
		e^{2C_K|V| (T-t)} \|\mu_2-\mu_{{1,m}}\|_{L^1(\rr)}.
	\end{align}
	by independence of $\mu_{{1,m}},\mu_2$ with respect to $v$.  Plug the above into \eqref{eq:gradMdifference}  and estimate $f$ by \eqref{eq:boundf} to obtain
	\begin{align*}
		\left|\frac{\partial (M_{{1,m}}-M_2)(K)}{\partial K_{r,i}}\right|& \leq 2 C_\phi \int_0^T e^{2C_K|V|t}\|\bar g_{{m}}(t)\|_{L^\infty(V;L^1(\rr))} \rd t \\
		&\leq 2C_\phi e^{2C_K|V| T}T \|\mu_2-\mu_{{1,m}}\|_{L^1(\rr)}.
	\end{align*}
	Since every entry $(r,i)$  converges, the gradient difference  vanishes $\|\grd_K M_{{1,m}}(K)- \grd_K M_2(K) \|_{{F}} \to 0$ as $m\to \infty$.
	
	We utilize this fact to show the degeneracy of the Hessian. Noting:
	\begin{align*}
		H_K\C(K_\star) 
		= \underbrace{\left[\sum_{l=3}^{2R} \grd M_l\otimes \grd M_l + 2  \grd M_2\otimes \grd M_2 \right]}_{A}+\underbrace{\left[\grd M_{{1,m}}\otimes \grd M_{{1,m}} - \grd M_2\otimes \grd M_2\right]}_{B^{(m)}}\,.
	\end{align*}
	
	It is straightforward that the rank of $A$ is at most $2R-1$, so  { the $j$-th largest eigenvalue  $\lambda_j(A)=0$ vanishes for some $j$}. 
	Moreover, since $\|\grd_K M_{{1,m}}(K)- \grd_K M_2(K) \|_{{F}} \to 0$, we have $\|B^{(m)}\|_F \to 0$. Using the continuity of the minimal eigenvalue with respect to a perturbation of the matrix, {the $j$-th largest eigenvalue of $H_K\C(K_\star)$ vanishes}
	\begin{align*}
		{|\lambda_j(H_K\C(K_\star))| = |\lambda_j(H_K\C(K_\star)) - \lambda_j(A)| }\leq  \|B^{(m)}\|_F \to 0, \quad \text{as }m\to \infty\,.
	\end{align*}
\end{proof}

\subsubsection{{Closeness in the weak sense}} 
We now study the {more general } scenario of Theorem \ref{thm:illconditioning} {where $\mu_{{1,m}}$ converges weakly in $L^1$. Because we no longer obtain smallness of $\bar g$ in the strong sense as in \eqref{eq:gbarnorm}, {the} additional assumption of small measurement time {is} required to provide smallness in a weak sense. }

\begin{proposition}\label{prop:gradconv_singular}
	{Let $\mu_{{1,m}}\rightharpoonup\mu_2$ weakly in $L^1$.} Consider a small neighbourhood of $K_\star$ and let Assumption \ref{ass:all} hold.  Additionally, let the measurement time $T$  be chosen such that $(e^{T|V|C_K}-1)<1.$ 
	Then 
	\[
	\grd_K M_{{1,m}}(K)\to \grd_K M_2(K) \quad \text{  as } {m\to \infty} \text{ in the standard Euclidean norm}.
	\]
\end{proposition}

This proposition explains the breakdown of well-posedness presented in Theorem~\ref{thm:illconditioning} {for weakly convergent measurement test functions}. Since the proof for the theorem is rather similar to that of the first scenario, we omit it from here.

{Similar to the previous scenario, we need to show smallness of the gradient difference \eqref{eq:gradMdifference}. 
	This time, we have to distinguish two sources of smallness: For singular parts of the adjoint $\bar g_{{m}}$, the smallness of the corresponding gradient difference is generated by testing it on a sufficiently regular $f$ at close measuring locations. So it is small in the weak sense. The regular parts $\bar g_{{m}}^{{(>N)}}$ of $\bar g_{{m}}$ {represent the difference of $\bar g_{{m}}$ and its singular parts and} evolve form the {integral operator on the right hand side of \eqref{eq:adjoint}, which exhibits a diffusive effect. Smallness is obtained by adjusting the cut off regularity $N$.}
	
	Let us mention, however, that the time constraint is mostly induced for a technical reason. In order to bound the size of {the regular parts of the} adjoint solution, we use the plain Gr\"onwall inequality which leads to an exponential growth {that we counterbalance by a small measuring time $T$}.  
	
	{To put the above considerations into a mathematical framework}, we deploy the singular decomposition approach, and we are to decompose
	\begin{equation}\label{eq:decomposition}
		\bar g_{{m}} = \sum_{n=0}^N\bar g_{{m}}^{{(n)}} + \bar g_{{m}}^{{(>N)}},
	\end{equation}
	{where the regularity of $\bar g_{{m}}^{{(n)}}$  increases with $n$}. Here, we define $\bar g_{{m}}^{{(0)}}$ as the solution to
	\begin{align*}
		-\dt \bar g_{{m}}^{{(0)}} -v\cdot \grd_x \bar g_{{m}}^{{(0)}} & = -\sigma \bar g_{{m}}^{{(0)}}\,,\\
		\bar g_{{m}}^{{(0)}}(t=T, x,v)& = {\mu_2(x) - \mu_{{1,m}}(x)\,},
	\end{align*}
	for $\sigma(x,v):= \int_V K(x,v',v)\rd v'$, and $\bar g_{{m}}^{{(n)}}$ are inductively defined by
	\begin{align}\label{eq:gn}
		-\dt \bar g_{{m}}^{{(n)}} -v\cdot \grd_x \bar g_{{m}}^{{(n)}} & = -\sigma\bar g_{{m}}^{{(n)}} + \tilde \L(\bar g_{{m}}^{{(n-1)}})\,,\\
		\bar g_{{m}}^{{(n)}}(t=T, x,v)& = 0\,,\nonumber
	\end{align}
	where we used the notation $\tilde \L (\bar g_{{m}}):= \int K(x,v',v)\bar g_{{m}}(x,t,v')\rd v'$. The remainder $\bar g_{{m}}^{{(>N)}}$ satisfies
	\begin{align}\label{eq:g>N}
		-\dt \bar g_{{m}}^{{(>N)}} -v\cdot \grd_x \bar g_{{m}}^{{(>N)}} & = -\sigma\bar g_{{m}}^{{(>N)}} + \tilde \L(\bar g_{{m}}^{{(N)}} + \bar g_{{m}}^{{(>N)}})\,,\\
		\bar g_{{m}}^{{(>N)}}(t=T, x,v)& = 0\,.\nonumber
	\end{align}

	It is a straightforward calculation that
	\begin{equation}\label{eqn:spell_g_bar}
		\eqref{eq:gradMdifference}= {\sum_{n=0}^N \int_0^T\int_{I_r}f'\big(({\bar g_{{m}}^{{(n)}}})' - \bar g_{{m}}^{{(n)}}\big)\rd x\rd t}
		+ \int_0^T\int_{I_r}f'\big(({\bar g_{{m}}^{{(>N)}}})' - \bar g_{{m}}^{{(>N)}}\big)\rd x\rd t\,.
	\end{equation}
	We are to show, in the two lemmas below, that both terms are small when {$\mu_{{1,m}}\rightharpoonup \mu_2$}. To be more specific: 
	
	\begin{lemma}\label{lem:IllConditioning:fgn}
		Let $\mu_{1,m}\rightharpoonup \mu_2$ weakly in $L^1$ as ${m\to \infty}$ and let $\|K\|_\infty\leq C_K$. Then $\bar g_{{m}}^{{(n)}}\rightharpoonup 0$  weakly in $L^1([0,T]\times \rr\times V)$ as $m\to \infty$ for all  $n\in \mathbb{N}_0$.
	\end{lemma}
	
	{The remainder can be bounded similarly.}
	\begin{lemma}\label{lem:IllCond:Remainder}
		Under the assumptions of  Proposition  \ref{prop:gradconv_singular}, one has 
		\begin{align*}
			\left|\int_0^T\int_{I_r} f' \bar g_{{m}}^{{(>N)}}\rd x \rd t\right| \leq  T^2 |V|C_K C_\phi e^{2|V|C_KT} (e^{C_K|V|T}-1)^N C_\mu,
		\end{align*}
		which becomes arbitrarily small for large $N$.
	\end{lemma}
	{The proof for \Cref{lem:IllConditioning:fgn} exploits the smallness of $\bar g_{{m}}^{{(n)}}$ in a weak sense which is inherited from the final condition, whereas \Cref{lem:IllCond:Remainder} is based on the smallness of the higher regularity components of $\bar g_{{m}}$ in the small time regime where tumbling is not so frequent.}
	Since it is not keen to the core of the paper, we leave the details to \Cref{appendix:lemma3} in the supplementary materials. The application of the two lemmas gives Proposition \ref{prop:gradconv_singular}:
	\begin{proof}[Proof of Proposition \ref{prop:gradconv_singular}]
		Let $\eps >0$. Because $e^{C_K|V|T}-1<1$ by assumption, we can choose $N\in \mathbb{N}$ large enough such that $2 T^2 |V|C_K C_\phi e^{2|V|C_KT} (e^{C_K|V|T}-1)^N<\frac{\eps}{2}$.
		Furthermore, \Cref{lem:IllConditioning:fgn} suggests that  
		\[
		\left|\int_0^T\int_{I_r} f'((\bar g_m^{(n)})'- \bar g_m^{(n)})\rd x \rd t\right| <\frac{\eps}{2(N+1)}
		\]
		for $m$ sufficiently large, given that $ f' \in L^\infty $ and $\bar g_n \rightharpoonup 0$ weakly in $L^1$.
		Then with the triangle inequality and Lemmas \ref{lem:IllConditioning:fgn} and \ref{lem:IllCond:Remainder}, we obtain from \eqref{eqn:spell_g_bar} 
		\begin{align*}
			&\left|\frac{\partial (M_{{1,m}}-M_2) (K)}{\partial K_{r,i}}\right|\\
			&\leq   \sum_{n=0}^N\left| \int_0^T\int_{I_r} f'((\bar g_{{m}}^{{(n)}})'-\bar g_{{m}}^{{(n)}})\rd x \rd t\right| + \left|\int_0^T\int_{I_r}f'((\bar g_{{m}}^{{(>N)}})' - \bar g_{{m}}^{{(>N)}})\rd x \rd t\right|\\
			&\leq   {(N+1)} \frac{\eps}{2{(N+1)}}+ 2 T^2 |V|C_K C_\phi e^{2|V|C_KT} (e^{C_K|V|T}-1)^NC_\mu
			&\leq  \eps\, .
		\end{align*}
	\end{proof}

	\section{Experimental Design} \label{sec:ExpDesign}
	We now provide an explicit experimental setup that ensures well-posedness {with a minimal number of measurements $L=2R$}. Recalling that Proposition~\ref{prop:well-posedness} requires the positive-definite-ness of the Hessian term at $K_\star$, we are to design a special experimental setup that validates this assumption. We propose to use the following:
	\begin{design}
		\label{design}
		We divide the {experimental} domain $I=[a_0, a_R)$ into $R$ intervals $I=\bigcupdot_{r=1}^R I_r$ with  $I_r=[a_{r-1},a_r)$, and the center for each interval is denoted by ${a_{r-1/2}}:=\frac{a_{r-1}+a_r}{2}$. The spatial supports of the values $K_r(v,v')$ takes on the form of \eqref{eq:Kstepfunction}. {The design is:}
		\begin{itemize}
			\item \emph{initial condition} $\phi(x,v) = \sum_{r=1}^R\phi_{{r}}(x)$ is a sum of $R$ positive functions $\phi_{{r}}$ that are compactly supported in $a_{r-1/2}+[-d,  d]$ with $d<\min\left(\frac{a_{r}-a_{r-1}}{4}\right)$, symmetric and monotonously decreasing in $|x-a_{r-1/2}|$ (for instance, a centered Gaussian with a cut-off tail);
			\item \emph{measurement test functions} $\mu_{{l_i^r}} =  \bar C_\mu\mathds{1}_{[{x_{l_i^r}}-{d_\mu},{x_{l_i^r}}+{d_\mu}]}$, $i=1,2$, {for some $\bar C_\mu >0$,} centered around ${x_{l_i^r}:=a_{r-1/2}+ (-1)^iT}$ with $d_\mu\leq d$; 
			\item \emph{measurement time} $T$ such that 
			\begin{align}
				&T<\min\left((1-\delta)\frac{0.09}{C_K|V|},\min_r\left(\frac{a_r-a_{r-1}}{4}-\frac{d}{2}\right)\right)\label{eq:design:Tsmall}\\
				&\text{for }\quad \delta = (d+d_\mu)/T<e^{-TC_K|V|}.\label{eq:design:Tdelta}
			\end{align}
		\end{itemize}
	\end{design}
	
	\begin{remark} 
		{Note that this design requires a delicate balancing between $T$ and $d$ and $d_\mu$. Requirement \eqref{eq:design:Tsmall} prescribes that $T$ must not be too large. {This places the experiment in a regime where tumbling does not occur too frequently. Indeed, when the system tumbles too often, one only obtains an accumulation of influence of many tumbling events, making it difficult to separate out each parameter. The same strategy was deployed in our theoretical paper~\cite{HKLT2022singulardecomposition}. However, the specific bound for $T$ may not be optimal. In the proof of~\Cref{thm:ourDesignWellPosed}, only crude estimates were deployed. It is possible to relax the bound for $T$.}}
	\end{remark}
	
	This particular design of initial data and measurement is to respond to the fact that the equation has a characteristic and particles moves along the trajectories. {The measurement is set up to single out the information {which} we would like to reconstruct along the propagation.} {A} visualization of this design is plotted in Figure~\ref{fig:MotionBallisticPartsfg}.
	
	\begin{figure}[H]
		\centering
		\begin{tikzpicture}[scale=0.75]
			\begin{axis}[every axis plot post/.append style={
					mark=none,domain=-5:5,samples=50,smooth}, 
				axis x line*=bottom, 
				axis y line*=left, 
				enlargelimits=upper, 
				yticklabels={},
				xtick = {-5,-2,0,2,5},
				xticklabels={$a_{{0}}$,$x_1$, $a_{{1/2}}$, $x_2$, $a_{{1}}$}
				] 
				\addplot [dashdotted, thick, cyan]{0.9*gauss(0,0.2)};
				\addplot [dashed, thick, blue]{0.8*gauss(-2,0.2)};
				\addplot [densely dotted, thick, blue]{0.8*gauss(2,0.2)};
				\addplot [dashed, thick, orange]
				coordinates { (0.3,0.5) (-0.3,0.5)};
				\addplot [dashed, thick, orange]
				coordinates { (0.3,0.5) (0.3,0)};
				\addplot [dashed, thick,orange]
				coordinates { (-0.3,0.5) (-0.3,0)};
				\addplot [densely dotted,thick, orange]
				coordinates { (-4.3,0.5) (-3.7,0.5)};
				\addplot [densely dotted, thick,orange]
				coordinates { (-4.3,0.5) (-4.3,0)};
				\addplot [densely dotted, thick,orange]
				coordinates { (-3.7,0.5) (-3.7,0)};
				
				\addplot [dashdotted,thick, red]
				coordinates { (-2.3,0.4) (-1.7,0.4)};
				\addplot [dashdotted,thick, red]
				coordinates { (-2.3,0.4) (-2.3,0)};
				\addplot [dashdotted,thick,red]
				coordinates { (-1.7,0.4) (-1.7,0)};
			\end{axis}
			\draw [-stealth, densely dotted, cyan](3.4,2.5) -- (4.1,2.5);
			\draw [stealth-, dashed, cyan](2.1,2.5) -- (2.8,2.5);
			\draw [-stealth, densely dotted, orange](0.8,0.9) -- (1.5,0.9);
			\draw [stealth-, dashed, orange](2.1,0.9) -- (2.8,0.9);
			\draw (7,-0.2) node {$x$};
			\draw[Bar-Bar, red] (1.66,-0.5)--(2.06,-0.5);
			\draw  [red](1.9,-0.8) node {\tiny $2d_\mu$};
			\draw[Bar-Bar, cyan] (2.7,-0.5)--(3.5,-0.5);
			\draw  [cyan](3.1,-0.75) node {\tiny $2d$};
		\end{tikzpicture}
		\caption{Motion of the ballistic parts  \textcolor{cyan}{$f^{{(0)}}(t=0,v)$} (cyan, dashdotted) to \textcolor{blue}{$f^{{(0)}}(t=T, v=+1)$} (blue, dotted) and \textcolor{blue}{$f^{{(0)}}(t=T, v=-1)$} (blue, dashed) and \textcolor{orange}{$g_{{1}}^{{(0)}}(t=0, v=+1)$} (orange, dotted) and \textcolor{orange}{$g_{{1}}^{{(0)}}(t=0, v=-1)$} (orange, dashed) to \textcolor{red}{$g_{{1}}^{{(0)}}(t=T,v)$} (red, dashdotted), compare also \eqref{eq:ballisticfg}.
		}
		\label{fig:MotionBallisticPartsfg}
	\end{figure}
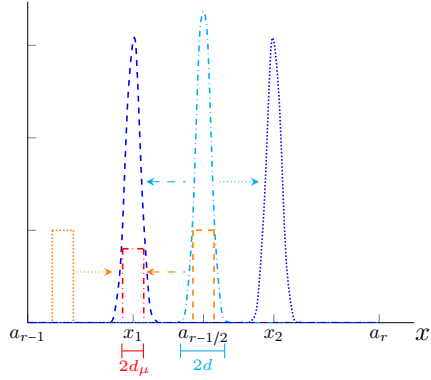
	
	Under this design, we have the following proposition:
	\begin{proposition}\label{prop:Designdecouple}
		The design \ref{design} decouples the reconstruction of $K_r$. To be more specific, recall~\eqref{eq:K_vec}
		\[
		K = [K_r]\,,\quad\text{with}\quad K_r=[K_{r,1},K_{r,2}]\,.
		\]
		The Hessian $H_K\C$ has a block diagonal structure with each of the blocks is a $2\times 2$ matrix given by the Hessian $H_{K_r}\C$. 
	\end{proposition}
	\begin{proof}
		By the linearity of \eqref{eq:chemotaxis}, {its} solution  $f=\sum_{s=1}^R f_{{s}}$ decompose{s} into solutions $f_{{s}}$ of \eqref{eq:chemotaxis} with initial conditions $\phi_{{s}}$. By construction of $T$ and the constant speed of propagation $|v| = 1$, the spatial supports of {the} $f_{{s}}$ 
		are
		fully contained  in $I_s$ 
		for all $t\in [0,T], v\in V$. As such, only $f_{{r}}$  
		carr{ies} information about $K_r$, and no information for other $K_s$ with $s\neq r$. {Because $f_{{r}}$ is only measured by measurements $M_{l^r_i}$, $i=1,2$, only the gradients of these measurements can attain a non-zero value corresponding to the partial derivative with respect to $K_r$.}
	\end{proof}
	
	This not only makes boundary conditions superfluous, but also translates the problem of finding a $2R$ valued vector $K$ into $R$ individual smaller problems of finding the two-constant pair $(K_{r,1},K_{r,2})$ within $I_r$. This comes with the cost of prescribing very detailed measurements depending on the experimental scales $I_r$ and $d$, but opens the door for parallelized computation.
	
	Furthermore, under mild conditions, this design ensures the local reconstructability of the inverse problem.
	\begin{theorem}\label{thm:ourDesignWellPosed} {Let Assumption \ref{ass:all} hold.}
		Given the Hessian $H_KM_l(K)$ is bounded {in Frobenius norm} in a neighbourhood of $K_\star$, 
		{then} Design \ref{design} generates a locally well-posed optimization problem \eqref{eq:loss}.
	\end{theorem}
	The proof is layed out in the subsequent subsection \ref{ssec:DesignWellPosed}. {As it relies on a perturbative argument, only local reconstructability can be proven. However, numerical experiments - such as those displayed in \Cref{fig:lossfunctions} and \Cref{fig:well-posedness} - exhibited reconstructability in a wide range of  parameters values, pointing towards global reconstructability with this design.}
	
	\begin{remark}\label{rem:SimilarSD}
		{Design \ref{design} shares similarities with the theoretical reconstruction setting in~\cite{HKLT2022singulardecomposition}: local  reconstructions are based on measurements close to the considered location and small time information is needed. In particular, the pointwise reconstruction of $K$ in~\cite{HKLT2021multiscale} relies on a sequence of experiments where} the measurement time asymptotically vanishes and the measurement location 
		gets close to the initial location. The situation is also seen here. As we refine the discretization for the underlying $K$-function using higher dimensional vector, measurement time has to be shortened to honor the refined discretization. However, we should also note the difference. In~\cite{HKLT2022singulardecomposition}, we studied the problem in higher dimension and thus explicitly excluded the ballistic {part of the data from the measurement}{, since it can only provide information on $\sigma$. Because there exist only two directions $V=\{\pm 1\}$ in 1D, there is a one-to-one correspondence of $\sigma$ and $K$ and ballistic data is sufficient for the reconstruction.}
	\end{remark}

	\subsection{Proof of Theorem \ref{thm:ourDesignWellPosed}} \label{ssec:DesignWellPosed}
	According to Theorem \ref{thm:well-posedness}, {it is sufficient} 
	to show $H_K\C(K_\star)>0$. As the Hessian attains a block diagonal structure (Proposition \ref{prop:Designdecouple}), we  study the $2\times 2$-blocks 
	\begin{equation}\label{eq:subHessian}
		H_{K_r}\C(K_\star) = \grd_{K_r}M_{l^r_1}(K_\star)\otimes  \grd_{K_r}M_{l^r_1}(K_\star) + \grd_{K_r}M_{l^r_2}(K_\star)\otimes  \grd_{K_r}M_{l^r_2}(K_\star).
	\end{equation} 
	Here the two measurements $M_{l^r_1}$, $M_{l^r_2}$ are inside $I_r$, and $\nabla_{K_r}=[\partial_{K_{r,1}},\partial_{K_{r,2}}]$. The positive definiteness of the full $H_K\C(K_\star)$ is  equivalent to the positive definiteness of each individual $H_{K_r}\C(K_\star)$. This is established in the subsequent proposition. 
	\begin{proposition}\label{prop:ourDesignWellPosed}
		{Let Assumption \ref{ass:all} hold.}
		{Then} Design \ref{design} produces a positive-definite  Hessian {block} $H_{{K_1}}\C(K_\star)$. 
	\end{proposition}
	According to~\eqref{eq:subHessian}, $H_{K_1}\C(K_\star)$ is positive definite if 
	\begin{equation}\label{eq:DesignWellPosed:GradEntriesIneq}
		\left|\frac{\partial M_1(K_\star)}{\partial K_{1,1}}\right| > \left|\frac{\partial M_1(K_\star)}{\partial K_{1,2}}\right|\quad  \text{ and } \quad \left|\frac{\partial M_2(K_\star)}{\partial K_{1,1}}\right| < \left|\frac{\partial M_2(K_\star)}{\partial K_{1,2}}\right|
	\end{equation}
	{hold true for the measurements $M_1,M_2$ corresponding to $K_1$.}
	Due to design symmetry, it is sufficient to study the first inequality. 
	{Consider the difference $\frac{\partial M_1(K_\star)}{\partial K_{1,1}}-\frac{\partial M_1(K_\star)}{\partial K_{1,2}}$. }
	Similar to~\eqref{eq:decomposition} and~\eqref{eqn:spell_g_bar}, we are to decompose the equation for $f$ and $g$ (\eqref{eq:chemotaxis} and  \eqref{eq:adjointM} respectively, with $K=K_\star$) into the ballistic parts $g_1^{{(0)}}$ and $f^{{(0)}}$ and the remainder terms. Namely, let $g_1^{{(0)}}$ and $f^{{(0)}}$ satisfy
	\begin{align}\label{eq:ballisticfg}
		\begin{cases}
			-\partial_tg_1^{{(0)}} - v\cdot \nabla_x g_1^{{(0)}}& = -\sigma g_1^{{(0)}}\\
			g_1^{{(0)}}(t=T,x,v)& = \mu_{{1}}(x)
		\end{cases}
		\quad\text{and}\quad\begin{cases}
			\partial_t f^{{(0)}} - v\cdot \nabla_x f^{{(0)}} & = -\sigma f^{{(0)}} \\
			f^{{(0)}} (t=0,x,v)& = \phi(x,v).  
		\end{cases}
	\end{align}
	Then the following two lemmas are in place with  $\mu_{{1}}(x)$ and $\phi(x,v)$ as in  {\Cref{design}}.
	\begin{lemma}\label{lem:DesignWellPosed:f0g0}
		In the setting of Proposition \ref{prop:ourDesignWellPosed}, for $(v,v') = (+1,-1)$, the ballistic part
		\begin{align}\label{eq:defBallisticGradEntryDifference}
			B:=& \left|\int_0^T \int_{I_1} f^{{(0)}}(v')(g_1^{{(0)}}(v')-g_1^{{(0)}}(v))\rd x \rd t\right|\\
			&- \left|\int_0^T \int_{I_1}f^{{(0)}}(v)(g_1^{{(0)}}(v)- g_1^{{(0)}}(v'))\rd x \rd t\right| \nonumber
		\end{align}
		satisfies
		\begin{align}\label{eq:DesignWellPosed:ballisticineq}
			&B\geq C_{\phi\mu}\left(e^{-TC_K|V|}T - (d_\mu +d)\right)>0,
		\end{align}
		where $C_{\phi\mu}=\int_{I_1}  \phi_{{1}}(x)\mu_{{1}}(-T+x)\rd x = \max_{a,b}\int_{I_1} \phi_{{1}}(x+a)\mu_{{1}}(-T+x+b)\rd x $ by construction of $ \phi_{{1}}, \mu_{{1}}$ {and $T$}.
	\end{lemma}
	
	At the same time, the remainder term is small.
	\begin{lemma}\label{lem:DesignWellPosed:S}
		In the setting of Proposition \ref{prop:ourDesignWellPosed}, the remaining scattering term
		\begin{align*}
			S:=  & \int_0^T\int_{I_1} f^{}(v')(g_{{1}}(v')-g_{{1}}(v))\rd x\rd t -  \int_0^T\int_{I_1} f^{{(0)}}(v')(g_1^{{(0)}}(v')-g_1^{{(0)}}(v))\rd x\rd t 
		\end{align*}
		is bounded uniformly in $(v,v')$ by 
		\begin{align}\label{eq:DesignWellPosed:Sbound}
			|S|\leq4 C_{\phi\mu}T\frac{C_K|V|T}{(1-C_K|V|T)^2}.
		\end{align}
	\end{lemma}
	
	Proposition \ref{prop:ourDesignWellPosed} is a corollary of Lemmas \ref{lem:DesignWellPosed:f0g0}, \ref{lem:DesignWellPosed:S}.
	\begin{proof}[Proof of Proposition \ref{prop:ourDesignWellPosed}]
		By the bounds obtained in lemmas \ref{lem:DesignWellPosed:f0g0}, \ref{lem:DesignWellPosed:S}, one has
		\begin{align*}
			&\left|\frac{\partial M_1(K_\star)}{\partial K_{1,1} } \right|- \left|\frac{\partial M_1(K_\star)}{\partial K_{1,2} } \right| \geq B-2|S|  \\
			&\geq
			C_{\phi\mu}\left(e^{-TC_K|V|}T - {(d_\mu+d)}\right)-  8 C_{\phi\mu}T\frac{C_K|V|T}{(1-C_K|V|T)^2}
			\\
			&\geq C_{\phi\mu} T\left(1- TC_K|V|-\delta - 8\frac{0.09(1-\delta)}{(1-0.09)^2}\right),
		\end{align*}
		{where the estimate  $e^{-TC_K|V|}\leq 1-TC_K|V|$ was used to derive the last line. This holds due to smallness of $TC_K|V|<0.09(1-\delta)<1$ by construction \eqref{eq:design:Tsmall}--\eqref{eq:design:Tdelta}, which also provides  positivity of the emerging term.} 
		In total, this shows the first part of inequality \eqref{eq:DesignWellPosed:GradEntriesIneq}. As the second part can be treated in analogy, it follows that $H_{K_1}\C(K_\star)$ is positive definite. 
	\end{proof}
	Finally, Theorem \ref{thm:ourDesignWellPosed} is a direct consequence of Proposition \ref{prop:ourDesignWellPosed}.
	\begin{proof}[Proof of Theorem \ref{thm:ourDesignWellPosed}]
		Repeated  application of the arguments to all $H_{K_r}\C(K_\star)$, $r=1,...,R$, shows that $H_K\C(K_\star)>0$. By the assumption of boundedness of the Hessian $H_KM_l(K)$ in a neighbourhood of $K_\star$,   \Cref{thm:well-posedness} proves local well-posedness of the inverse problem.
	\end{proof}

	The proofs for the Lemmas~\ref{lem:DesignWellPosed:f0g0}-\ref{lem:DesignWellPosed:S} are rather technical and we leave them to \Cref{ap:DesignWellposed:ProofSsmall} in the supplementary materials. Here we only briefly present the intuition. According to Figure \ref{fig:MotionBallisticPartsfg}, $f^{{(0)}}({v'=-1})$ and $g_1^{{(0)}}({v'}=-1)$ have a fairly large overlapping support, whereas $g_1^{{(0)}}(v=+1)$ overlaps with $f^{{(0)}}(v'=-1)$ {and $g_1^{{(0)}}(v'=-1)$  with $f^{{(0)}}(v=+1)$  only for a short time spans $T\approx T$ and $T\approx 0$ respectively}. Recalling~\eqref{eq:defBallisticGradEntryDifference}, we see the negative components of the term $B$ are small, making $B$ positive overall. The smallness of $S$ is a result of small measurement time $T$.
	
	\section{Numerical experiments} \label{sec:NumExp}
	As a proof of concept for the prediction given by the theoretical results in Section \ref{sec:theory}, we present some numerical evidence.
	
	An explicit finite difference scheme is used for the discretization of \eqref{eq:chemotaxis} and \eqref{eq:adjoint}. In particular, the transport operator is discretized by the Lax-Wendroff method and the operator $\K$ is treated explicitly in time. The scheme can be shown to be consistent and stable when $\Delta t \leq \min(\Delta x, C_K^{-1})$, and thus it converges according to the Lax-Equivalence theorem. More sophisticated solvers for the forward model~\cite{FilbetYang_NumericsChemotaxis_2014} can be deployed when necessary. Also, when a compatible solver \cite{ApelFlaig_2012_OtDandDtOshouldcommute} for the adjoint equation exists, these pairs of  solvers can readily be incorporated in the inversion setting.

	All subsequent experiments were conducted with noise free synthetic data $y_l= M_l(K_\star)$ that was generated by a forward computation with the true underlying parameter $K_\star$. 
	
	\subsection{Illustration of well-posedness}
	In Section \ref{sec:ExpDesign}, it was suggested {that} a specific design of initial data and measurement mechanism can provide a successful reconstruction of the kernel $K$, and that the loss function is expected to be {locally} strongly convex. We observe it numerically as well. In particular, we set $R=20$ and use Gaussian initial data, and plot the (marginal) loss function in Figure \ref{fig:lossfunctions}.
	\begin{figure}
		\centering
		\includegraphics[height = 4cm]{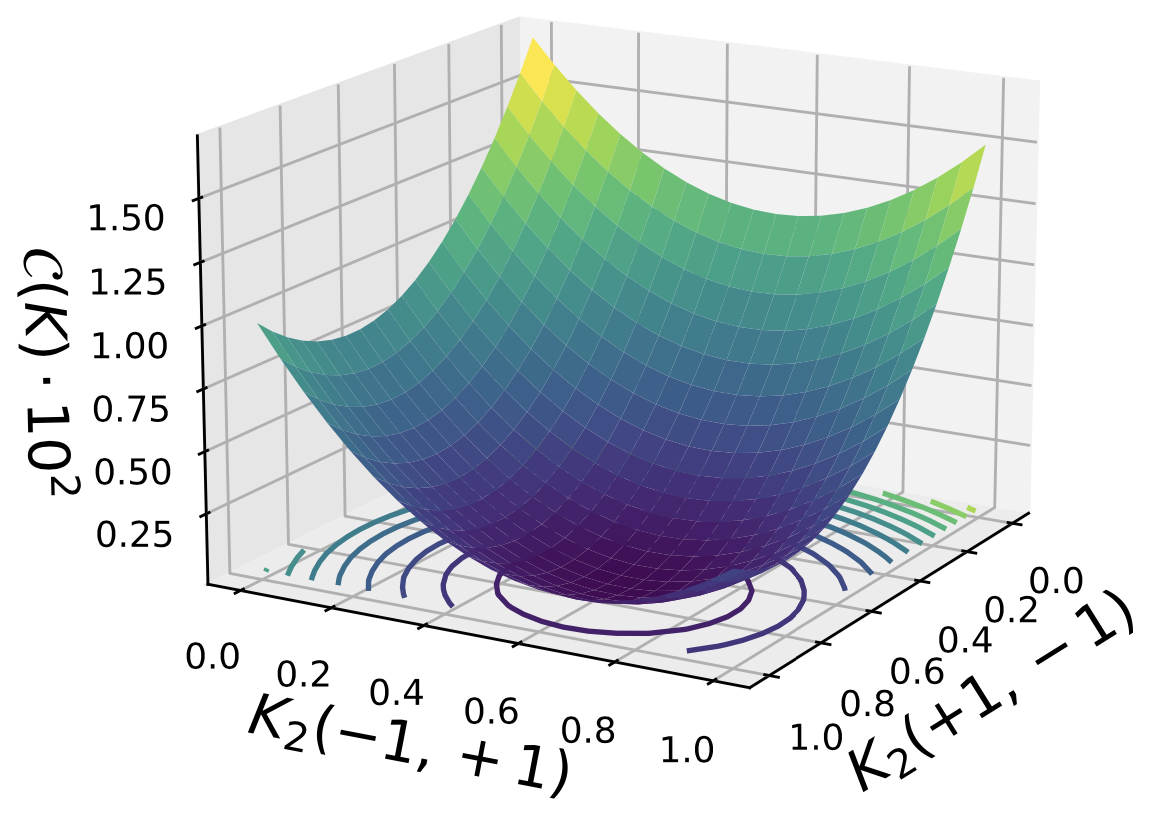}
		\includegraphics[height = 4cm]{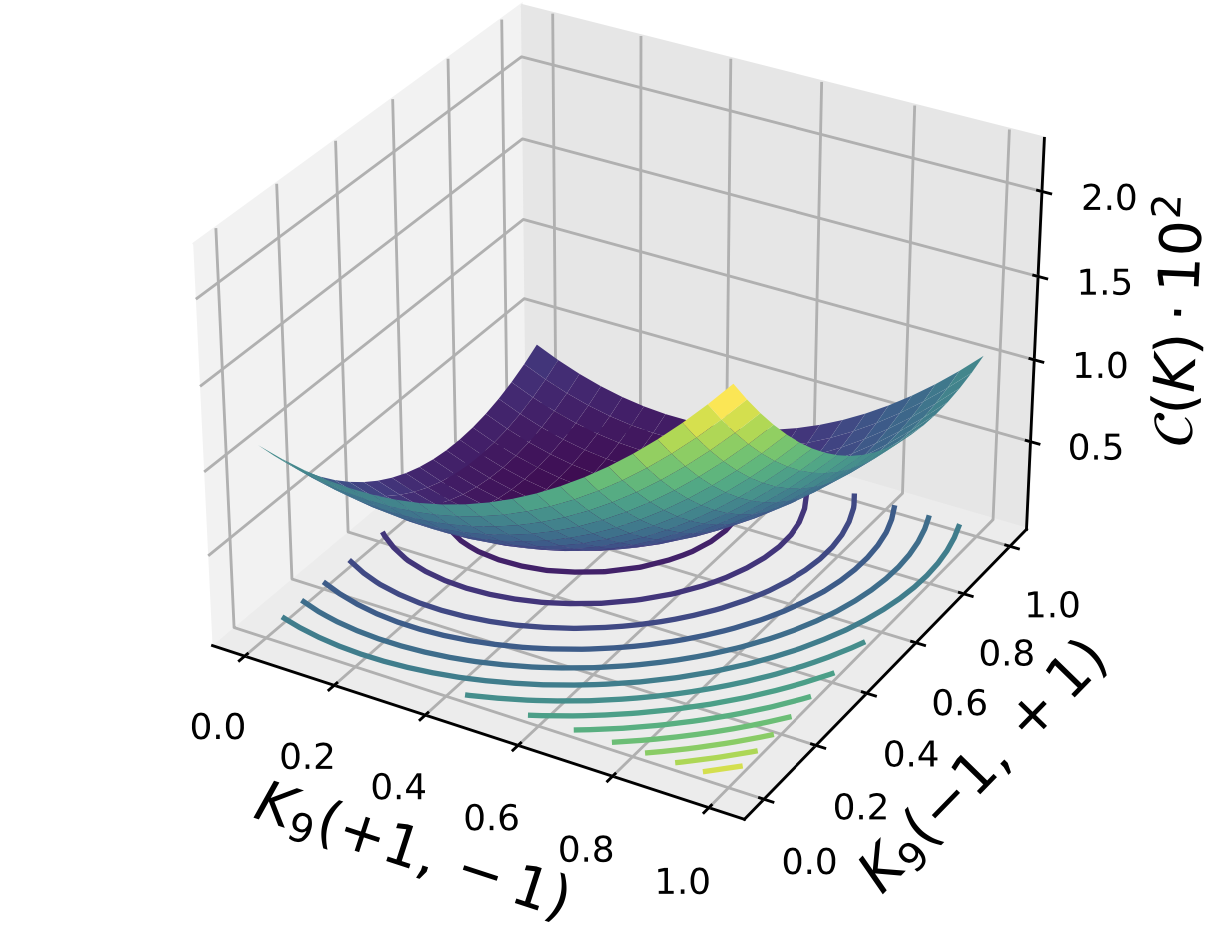}
		\includegraphics[height = 4cm]{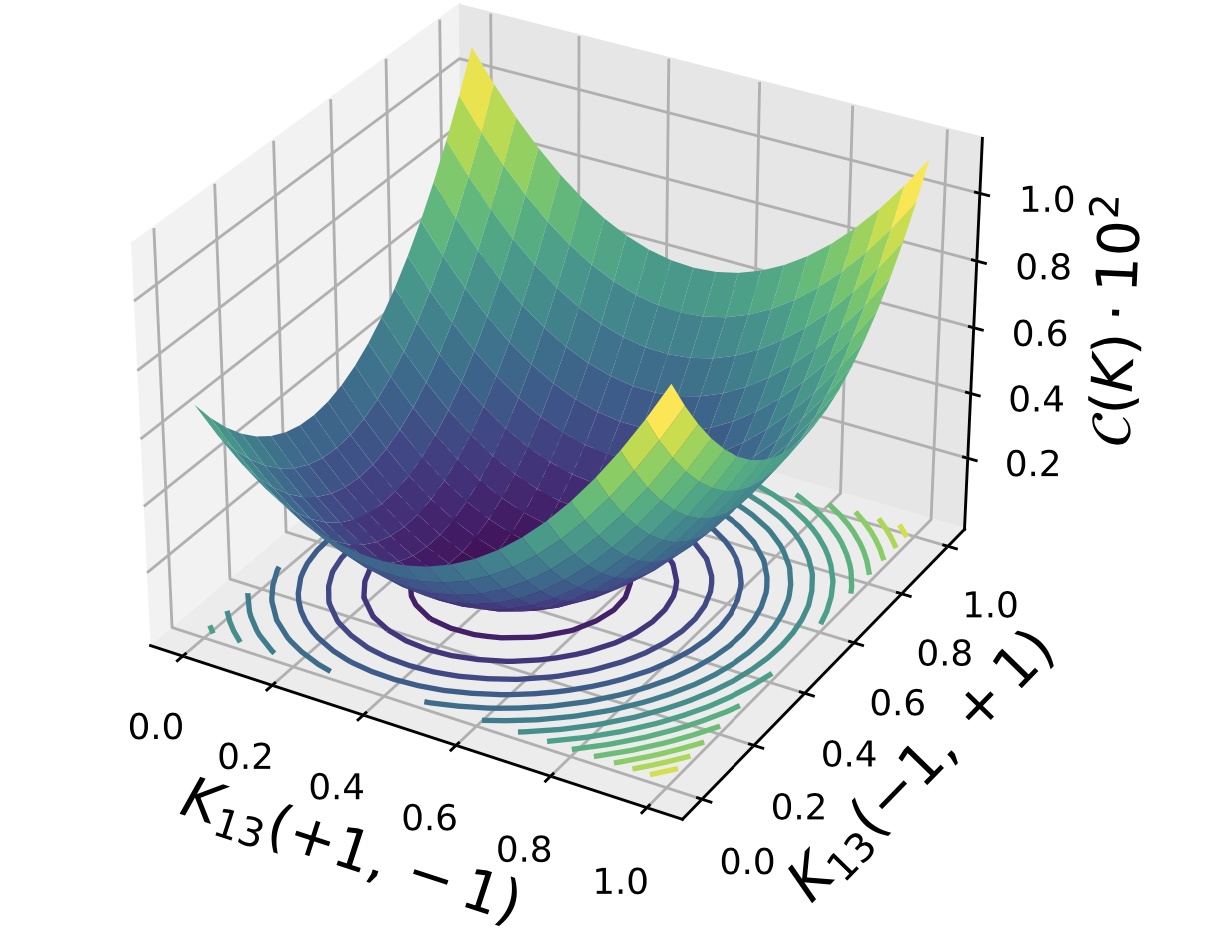}
		\includegraphics[height = 4cm]{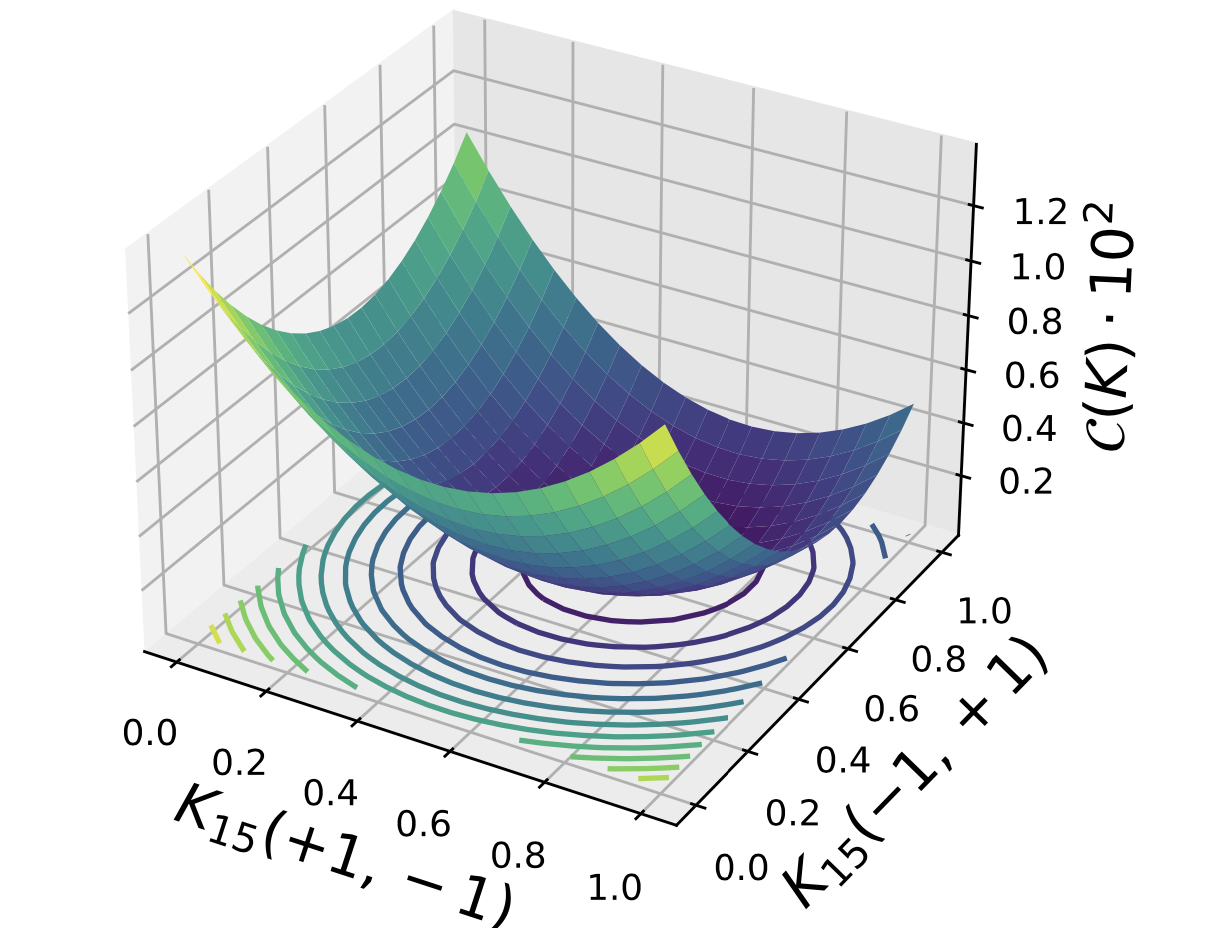}
		\caption{(Marginal) loss functions $\C(K)$ for $R=20$: For a fixed $r\in\{2,9,13,15\}$, we plot $\C$ as a function of $K_r$ with all $K_{s\neq r}$ set to be the ground-truth $(K_\star)_s$.}\label{fig:lossfunctions}
	\end{figure}
	Figure \ref{fig:Inversion} depicts the convergence of some parameter values $K_r(v,v')$ in this scenario against the corresponding loss function value. An exponential decay of the loss function, as expected from theory~\cite[Th.3]{PolyakShcherbakov_2017_Optimization}, can be observed.
	\begin{figure}[H]
		~\includegraphics[width = 8cm]{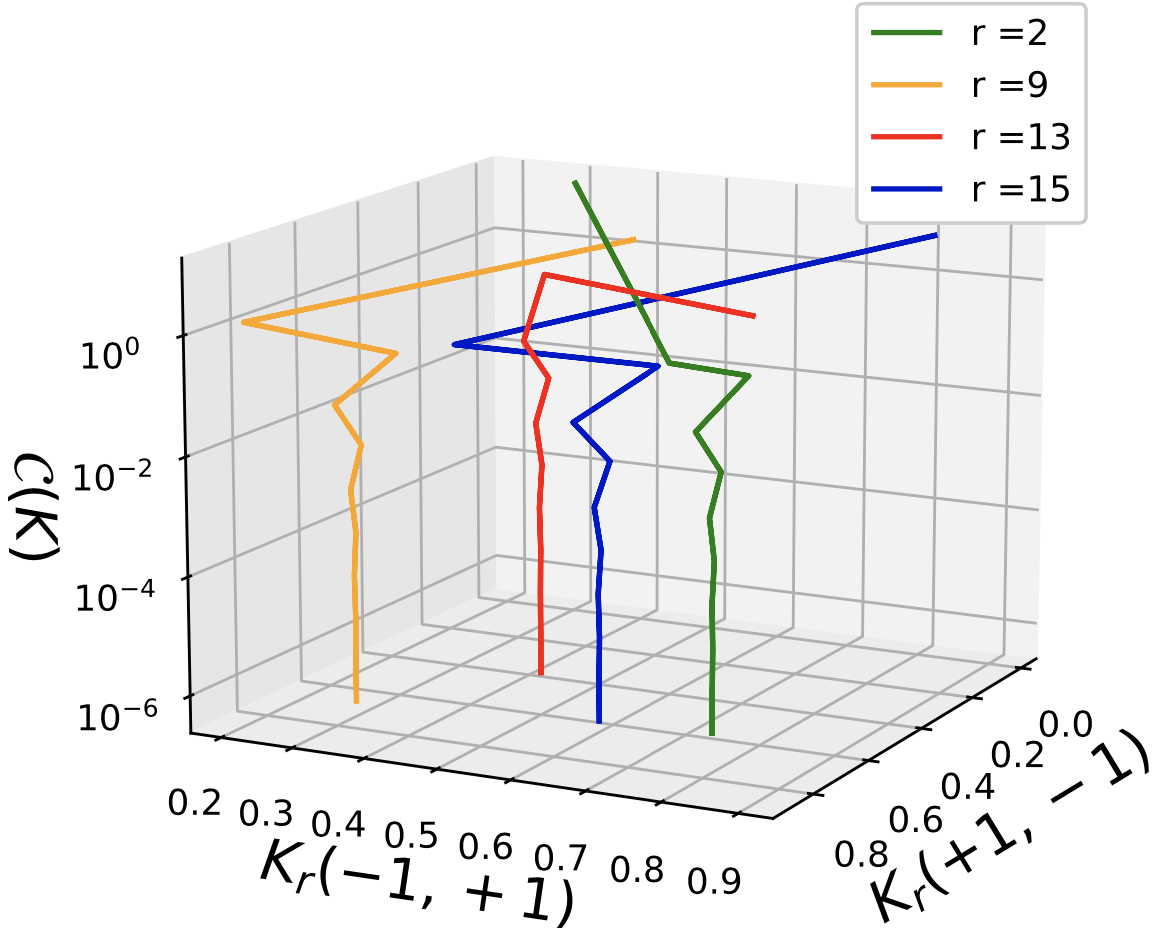}
		\centering
		\caption{Convergence of the parameter values $K_r(v,v')$ from \eqref{eq:Kstepfunction} for $r=2,9,13,15$ to the ground truth as the cost function converges.} 
		\label{fig:Inversion}
	\end{figure}

	The strictly positive-definiteness feature persists in a small neighborhood of the optimal solution $K_\star$. This means adding a small perturbation to $K_\star$, the minimal eigenvalue of the Hessian matrix $H_K\C(K)$ stays above zero. In Figure \ref{fig:well-posedness} we present, for two distinct {experimental} 
	setups, the minimum eigenvalue as a function of the perturbation to $K_r(v,v')$.  In both cases, the green spot ({located at }the ground-truth) is positive, and it enjoys a small neighborhood where the minimum eigenvalue is also positive, as predicted by Theorem \ref{thm:well-posedness}. In the right panel, we do observe, as one moves away from the ground-truth, the minimal eigenvalue takes on a negative value, suggesting the loss of convexity. This numerically verifies that the well-posedness result in Theorem \ref{thm:well-posedness} is local in nature. 
	The panel on the left deploys the experiment design provided by Section~\ref{sec:ExpDesign}. The simulation is ran over the entire {parameter} domain of $[0,1]^2$ and the positive-definiteness stays throughout the domain, hinting the proposed experimental design \ref{design} can potentially be globally well-posed. {To generate the plots, a simplified setup with $R=2$  was considered.} 
	
	\begin{figure}[H]
		~\includegraphics[height = 3.3cm]{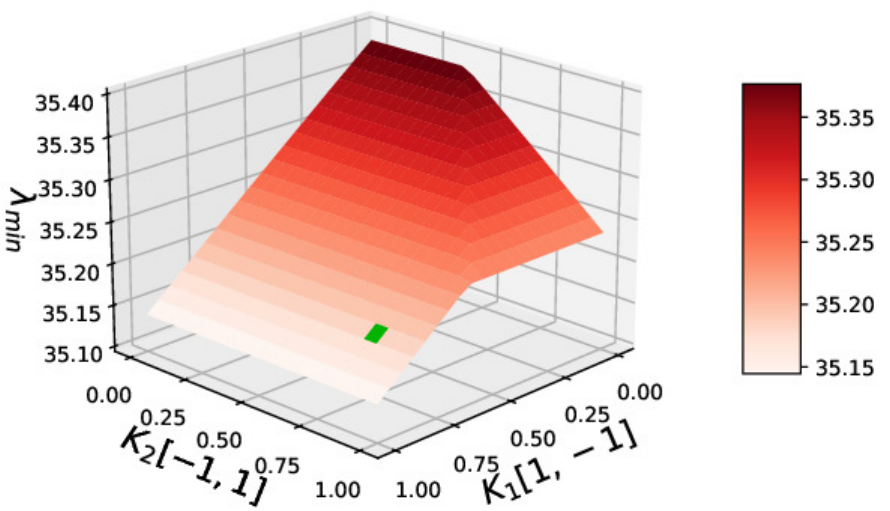}
		~\includegraphics[height = 3.3cm]{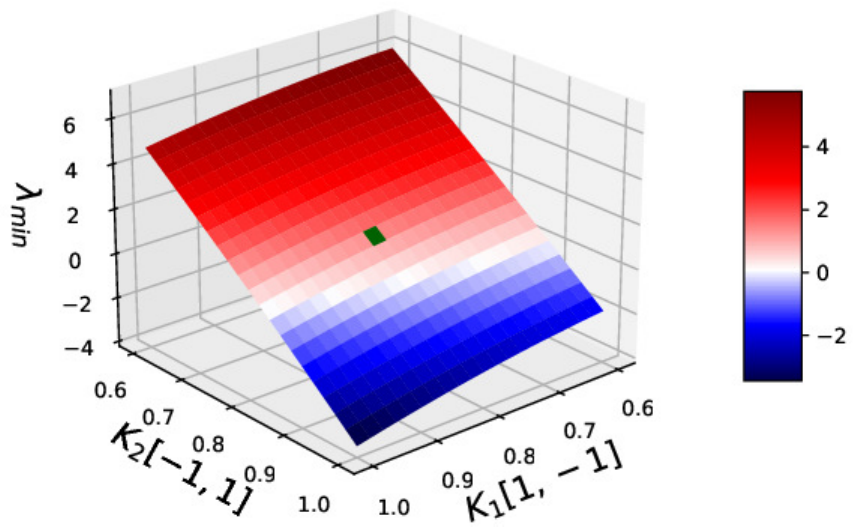}
		\centering
		\caption{Minimal eigenvalues of the Hessian $H_K\C(K)$ around the true parameter $K_\star$ for two experimental designs. We perturb $K$ by changing values in $K_1(1,-1)$ and $K_2(-1,1)$. The ground-truth is marked green in both plots.} 
		\label{fig:well-posedness}
	\end{figure}

	\subsection{Ill-conditioning for close measurement locations}
	We now provide numerical evidence to reflect the assertion from~\ref{ssec:IllCond}. In particular, the strong convexity of the loss function would be lost if measurement location $x_1$ becomes close to $x_2$.

	We summarize the numerical evidence in Figure \ref{fig:IllConditioning:reconstructions}. Here we still use $R=20$ and {fix the} initial data {as in \Cref{design},} but vary the detector positions. To be specific, we assign values to $x_1$ using {$\{x_{{1,0}}=a_{1/2}-T\,,x_{{1,1}}= a_{1/2}+\frac{T}{2}\,,x_{{1,2}}=a_{1/2}+\frac{4}{5}T\,, x_{{1,3}} = x_2 = a_{1/2}+T\}$}. As the superindex grows, $x_1\to x_2$ with $x_{{1,3}}=x_2$ when the two measurements exactly coincide. {For $x_1=x_2$, the cost function is no longer strongly convex around the ground truth $K_\star$, as its Hessian is singular. The thus induced  vanishing learning rate $ \eta =\frac{2\lambda_{\min}}{\lambda_{\max}^2}$ was exchanged by the learning rate for $x_1 = x_{{1,2}}$ in this case to observe the effect of the gradient.  }
	
	{In the first, third and fourth panel of Figure~\ref{fig:IllConditioning:reconstructions}, we observe that the cost function as well as the parameter reconstructions for $K_9$ and $K_{15}$ nevertheless converge, 
		but convergence rates  slow down significantly comparing purple (for $x_{{1,0}}$), blue (for $x_{{1,1}}$), green  (for $x_{{1,2}}$) and orange (for $x_{{1,3}}$) due to smaller learning rates. The overlap of the parameter reconstructions for $x_1\in \{x_{{1,2}}\,, x_{{1,3}}\}$ is due to the coinciding choice of the learning rate and a very similar gradient for parameters $K_9, K_{15}$ whose information is not reflected in the  measurement in $x_1$.}
	
	{As parameter $K_1$ directly affects the measurement at $x_1$, Panel 2 showcases the degenerating effect  of the different choices of  $x_1$ on the reconstruction. Whereas convergence is still obtained in the blue curve (for $x_{{1,1}}$), the reconstructions of $K_1$ from measurements at $x_{{1,2}}$ (green) and $x_{{1,3}}$ (orange) clearly fail to converge to the true parameter value in black. This offset seems to grow with  stronger degeneracy in the measurements.}
	\begin{figure}[H]
		~\includegraphics[width = 6cm]{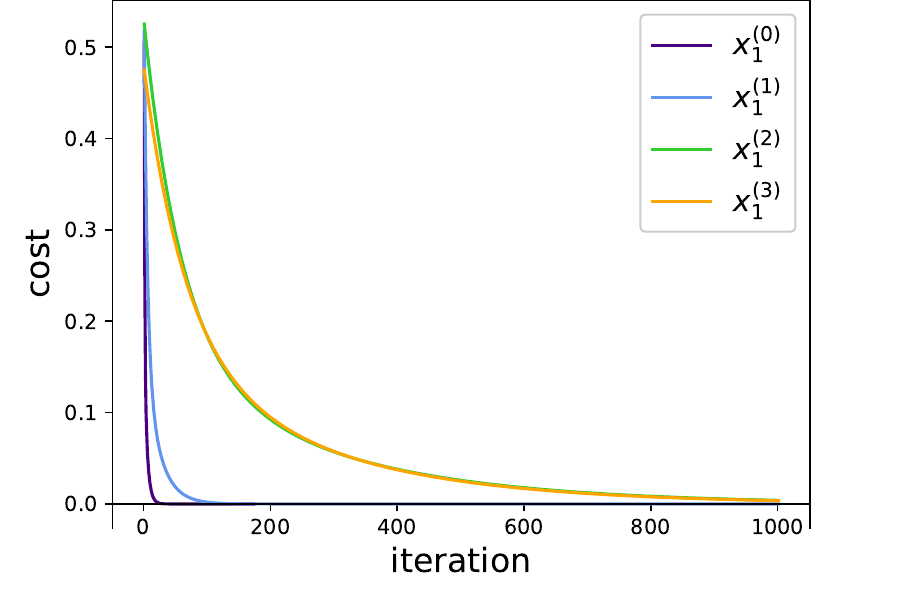}
		~\includegraphics[width = 6cm]{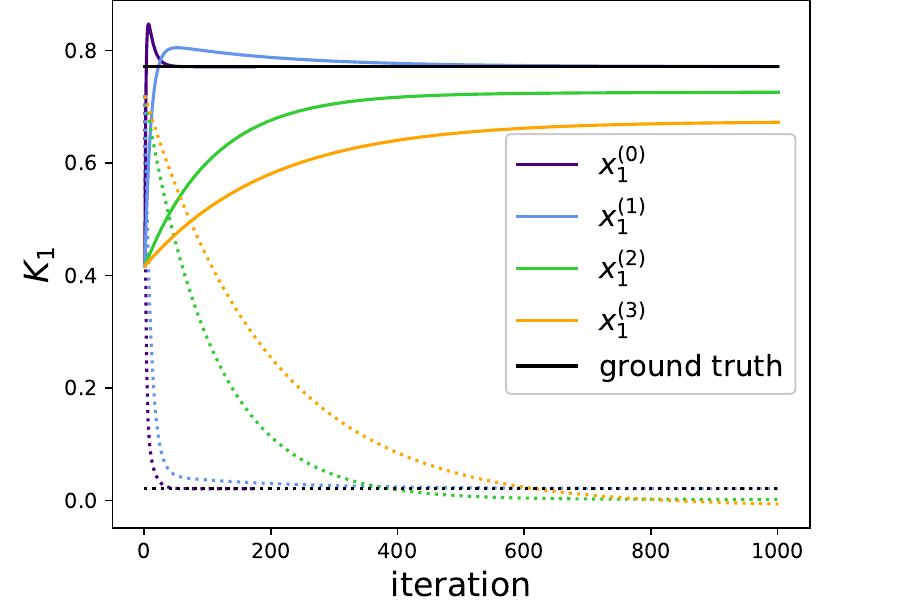}\\
		~\includegraphics[width = 6cm]{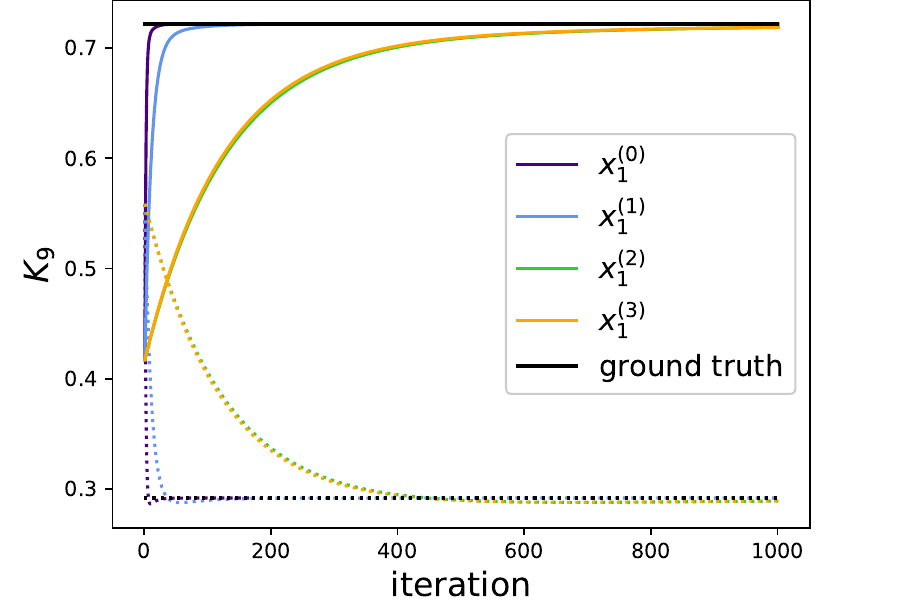}
		~\includegraphics[width = 6cm]{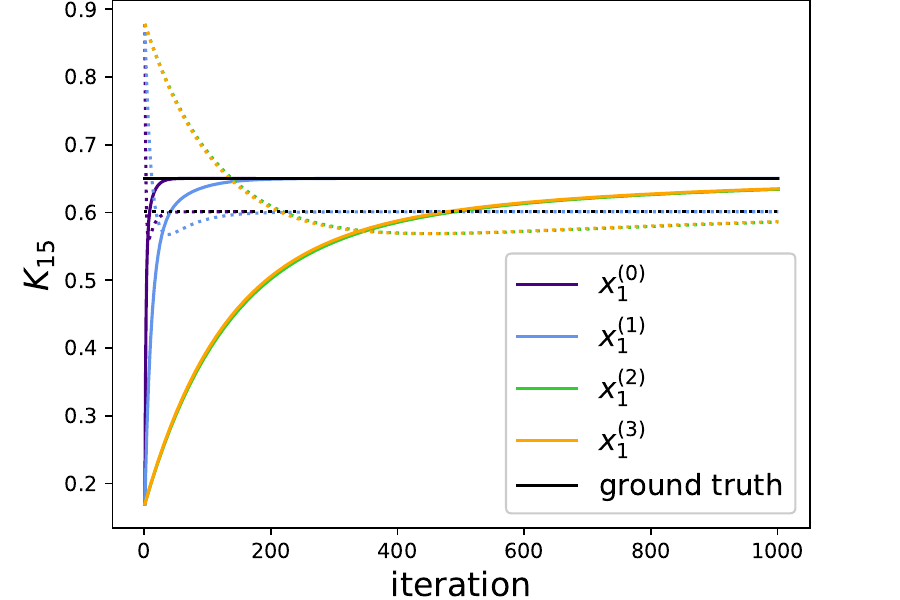}
		\centering
		\caption{Cost function and reconstructions of $K_r(+1,-1)$ (solid lines) and $K_r(-1,+1)$ (dotted lines) for $ r = 1,9,15$ and $R=20$ under different measurement locations for $x_1$ {given by }  
			$\{x_{{1,0}}=a_{1/2}-T\,,x_{{1,1}}= a_{1/2}+\frac{T}{2}\,,x_{{1,2}}=a_{1/2}+\frac{4}{5}T\,, x_{{1,3}} = a_{1/2}+T\}$ with $x_{{1,3}}=x_2$.}
		\label{fig:IllConditioning:reconstructions}
	\end{figure}

	\section{Discussion}\label{sec:Discussion}
	{As discussed in \cite{KineticParametersFromMacroMeasurements2,KineticParametersFromMacroMeasurements1}, to accurately extract tumbling statistics, it is necessary to track single-cell trajectories, which necessitates a low cell concentration and is constrained to shorter trajectories. This will result in insufficient statistical accuracy for reliable extraction of velocity jump statistics.}
	In this paper we present an optimization framework for the reconstruction of the velocity jump parameter $K$ in the chemotaxis equation \eqref{eq:chemotaxis} using velocity averaged measurements \eqref{eq:M} from the interior domain. {The velocity-averaged measurements do not require tracking single-cell trajectories, thus allowing for the measurement of higher cell density over a longer period of time. This may provide a new and reliable way of determining the microscopic statistics. }
	In the numerical setting when PDE-constrained optimization is deployed, depending on the experimental setup, the problem is can be either locally well-posedness or ill-conditioned. We further propose a specific experimental design that is adaptive to the discretization of $K$. This design decouples the reconstruction of local values of the parameter $K$ using the corresponding measurements. The design thus opens up opportunities to parallelization. As a proof of concept, numerical evidence were presented. They are in good agreement with the theoretical predictions.
	
	A natural extension of the results presented in the current paper is the algorithmic performance in higher {space} dimensions. The theoretical findings seem to apply in a straightforward manner, {and we are convinced that an adaptation of  \Cref{design} in analogy to \Cref{rem:SimilarSD} and \cite{HKLT2022singulardecomposition} could  provide well-posedness,} but details need to be evaluated. Numerically one can certainly also refine the solver implementation. For example, it is possible that higher order numerical PDE solvers that preserve structures  bring extra benefit. More sophisticated optimization methods such as the (Quasi-)Newton method or Sequential Quadratic Programming are possible alternatives for conducting the inversion~\cite{Burger_2002_SequentialProgramming,Haber_2000_(Gauss)NewtonMethods, Ren_ReviewNumericsTransportBasedImag,Smyletal_2021_QuasiNewton}. Furthermore, we adopted a {first optimize, then discretize} approach in this article. Suggested in~\cite{ApelFlaig_2012_OtDandDtOshouldcommute,Gunzburger_2002_Optimization, LiuWang_2017_nonCommutativeDtOOtD}, a {first discretize, then optimize} framework could be bring automatic compatibility of forward and adjoint solvers, but extra difficulties~\cite{Hinze2008OptimizationWP} need to be resolved. Error estimates for continuous in space ground truth parameters as in \cite{jin2021error} could help practitioners to select  a suitable space-discretization.

	Our ultimate goal is to form a collaboration between practitioners to solve the real-world problem related to bacteria motion reconstruction~\cite{Le_2002_Chemotaxis(macro)InBioreactors}. To that end, experimental design is non avoidable. A class of criteria proposed under the Bayesian perspective shed light on this topic, see~\cite{Alexanderian_2021_ReviewBayesianOptimalExperimentalDesign} and references therein. In our context, it translates to whether the design proposed in Section~\ref{sec:ExpDesign} satisfies these established optimality criteria.

\newpage
\appendix
	\section{Derivation of the gradient \eqref{eq:gradC}}
\label{sec:App:Opt}
This section justifies formula \eqref{eq:gradC} for the gradient of the cost function $\C$ with respect to  $K$. 
Let us first introduce some notation: Denote by 
\begin{align*}
	\mathcal{J}(f):=\frac{1}{2L} \sum_{l=1}^L \left(\int_\rr \int_Vf(T,x,v)\rd v\ \mu_l(x)\rd x -y_l\right)^2 
\end{align*}
the loss for {$f \in \mathcal{Y}= {C^0([0,T];L^\infty(\rr\times V)\cap L^1(\rr;L^\infty(V)))}$. 
	{According to the semigroup theory, see e.g. \Cref{lem:existencef}, $f$ is the mild solution to \eqref{eq:chemotaxis} if and only if it satisfies
		\begin{equation}\label{eqn:mild}
			f_K(t,x,v)  =\phi(x-vt, v) + \int_0^t \K_K(f)(s,x-v(t-s), v)\rd s\,.
		\end{equation}
		where we use the subindex $K$ to indicate $f$'s dependence on $K$.} With this notation: $ \C(K) := \mathcal{J}(f_K)$ in the notation of \eqref{eq:M}.
	
	We now deploy the Lagrangian formulation. Denote $h$ the Lagrangian multiplier, the PDE constrained optimization problem \eqref{eq:loss} reformulates as:
	\begin{align*}
		{\L(K,f,h)= \J(f) + \left \<h,f(t,x,v)  -\phi(x-vt, v) - \int_0^t \K_K(f)(s,x-v(t-s), v)\rd s\right \>}
	\end{align*}
	{for $f\in \mathcal{Y}$.} {$h$ as the Lagrangian multiplier lives in the dual space, and in this context $h\in L^1(\rr \times V;\mathcal{M}([0,T]))$, where $\mathcal{M}([0,T])$ denotes the space of finite Borel  measures on $[0,T]$. We should note the very weak regularity assumption on $h$ in $t$ variable, so the bracket notation $\<\cdot, \cdot\>$ actually means $\<f,h\>:= \int_\rr\int_V \int_0^T f(t,x,v) \d h_{x,v}(t)\rd x \rd v$.} 
	
	For  $f=f_K$, our cost function {is} $\C(K) = \J(f_K)=\L(K,f_K, g,\lambda) $   and
	\begin{displaymath}
		\frac{\d\C (\hat K)}{\d K}
		=  \frac{\partial \L }{\partial K} \bigg|_{\substack{K = \hat{K},\\f=f_{\hat K}}}
		+  \frac{\partial \L}{\partial f}\bigg|_{\substack{K = \hat{K},\\f=f_{\hat K}}}  \frac{\partial f_K}{\partial K} \bigg|_{K = \hat K}\,.
	\end{displaymath}
	Suppose we can find a specific value ${h = \hat h}$ such that 
	\begin{equation}\label{app:dLdf=0}
		\frac{\partial \L}{\partial f} \bigg|_{\substack{K = \hat{K},f=f_{\hat K}, {h=\hat h}}} = 0\,,
	\end{equation}
	then the complication of the formula above is significantly reduced, to:
	\begin{equation}\label{eqn:dCdK}
		\frac{{\partial}\C(\hat{K})}{{\partial} K_r} =  \frac{\partial \L}{\partial K_r}\bigg|_{\substack{K = \hat{K},f=f_{\hat K}, {h=\hat h}}}\,.
	\end{equation}
	In what follows below, we find the condition for $\hat h$ so that~\eqref{app:dLdf=0} is satisfied, and give an explicit computation to~\eqref{eqn:dCdK}.
	
	As a preparation, we first interpret the function space. More specifically we will examine the following term:
	\begin{equation}\label{eqn:special}
		{\left\<h, \int_0^t\K_K(f)(s,x-v(t-s),v)\rd s\right\>}\,.
	\end{equation}
	Expand out the bracket, and define the shifted measure $\tilde h_{x,v}^{(s)}(A) = h_{x,v}(A+s)$, we have this equals to
	\begin{align*}
		\eqref{eqn:special}=\int_\rr\int_V\int_0^T\int_0^{T-s} \K_K(f)(t,x,v)  \rd\tilde h_{x+vs,v}^{(s)}(t) \rd s \rd v \rd x
		=:\< \K_K(f), G^h \>\,,
	\end{align*}
	where we defines $G^h\in L^1(\rr\times V;\mathcal M([0,T]))$ through the weak form, in the sense that
	\[
	\int_0^T \phi(t) \rd G^h_{x,v}(t) :=   \int_0^T \int_0^{T-s} \phi(t) \rd \tilde h_{x+vs, v}^{(s)}(t)\rd s
	\]
	for any measurable $\phi$. It can be shown that $G_{x,v}^h$ is absolutely continuous with respect to the Lebesgue measure $\rd t$ for almost all $x,v$, so it can be represented by a density, namely there exists $g^h\in L^1(\rr\times V \times[0,T])$ so that $\rd G^h_{x,v}(t) = g^h(t,x,v) \rd t$. Denote the adjoint operator to $\K_K$ as $\tilde\K_K$, we further simplify~\eqref{eqn:special} to:
	\begin{equation}\label{eqn:equivalence}
		\eqref{eqn:special}=\<\K_K(f), G^h\> = \<f,\tilde \K_K(g^h)\>\,.
	\end{equation}
	In our context, the adjoint operator can be explicitly computed:
	\begin{equation}\label{eqn:adj}
		\tilde\K_K(g) := \int_VK(x,v',v)(g(t,x,v')-g(t,x,v))\rd v'\,.
	\end{equation}
	With these terms prepared, to compute~\eqref{eqn:dCdK}, we combine~\eqref{eqn:equivalence} and~\eqref{eqn:adj}, and have:
	\begin{align}\label{eqn:gradient}
		\frac{\partial \C(\hat K)}{\partial K_r} = \int_0^T\int_{I_r} f_{\hat K}(t,x,v') ( g^{\hat h}(t,x,v')- g^{\hat h}(t,x,v))) \rd x \rd t.
	\end{align}
	with the correct $\hat h$.
	
	To find this $\hat h$, we deploy~\eqref{app:dLdf=0}. Realizing:
	\begin{align*}
		\frac{\partial \mathcal L}{\partial f} 
		=& \frac{1}{L}\sum_{l=1}^L \left( \int_\rr\int_V f(T)\d v\ \mu_l\rd x-y_l\right) \frac{\partial }{\partial f }  \left\<\mu_l(x){\delta_T^-}(t), f\right\>
		+\frac{\partial }{\partial f } { \left\< f, h- \tilde\K_K (g^h) \right\>,}
		\nonumber
	\end{align*}
	where we used~\eqref{eqn:equivalence} and denotation $\delta_T^-$ \footnote{It is a version of the Dirac delta that allows $\int_0^T \phi(t) \delta_T^-(t)\rd t = \phi(T)$ for continuous in time functions $\phi$}. Setting this to be zero, we have:
	\begin{align*}
		{\hat h(t,x,v) =}  { \tilde\K_K (g^{\hat h})}
		{-  \frac{1}{L}\sum_{l=1}^L \left( \int_\rr\int_V f_{\hat K}(T)\d v\ \mu_l\rd x-y_l\right) \mu_l(x)\delta_T^-(t)}.
	\end{align*}
	With the same semigroup theory argument used in~\eqref{eqn:mild}, this means $g^{\hat h}$ attains regularity of $C^0([0,T];L^1(\rr\times V))$ as defined in:
	\begin{align*}
		{g_{\hat h}(t,x,v) =}
		&\ {-  \frac{1}{L}\sum_{l=1}^L \left( \int_\rr\int_V f_{\hat K}(T)\d v\ \mu_l\rd x-y_l\right) \mu_l(x+v(T-t))}\\
		& \ { +\int_t^T \tilde \K_{\hat K}(g_{\hat h}) (s,x+v(s-t),v) \rd s}\,,
	\end{align*}
	which in turns is the mild solution to the PDE:
	\begin{align*}
		&-\partial_t  g -v\cdot \nabla  g -\tilde{\K}_{\hat K}( g)=0, &\hspace{-.25cm}& (x,v,t)\!\in\! \rr\times V\times(0,T),\\
		&  g(t=T) =- \frac{1}{L}\sum_{l=1}^L \left(\!\int_\rr\!\int_V \! f_{\hat K}(z,T,v)\d v \ \mu_l(z)\rd z-y_l\!\right)\mu_l(x),&\hspace{-.25cm}&(x,v)\in \rr\times V\,.
	\end{align*}
	Note that since $ g^{{\hat h}}$ reflects the measurement procedure, it makes sense that $ g^{{\hat h}}(t=T)$ is isotropic in $v$. Inserting this back in~\eqref{eqn:gradient}, we arrive at the conclusion.

	\section{Some a-priori estimates}\label{sec:a_priori}
	
	By Assumption \ref{ass:all}, semigroup theory yields the existence of a {mild} solution to 
	\eqref{eq:chemotaxis}--\eqref{eq:chemotaxis_init}.
	
	\begin{lemma} \label{lem:existencef}
		Let Assumption \ref{ass:all}  
		hold and assume {that} $h\in {L^1}((0,T); L^\infty(\rr\times V){\cap L^1(\rr;L^\infty(V))}) $. Then there exists a {mild} solution 
		\begin{align}\label{eq:regulatityf}
			f\in 
			C^0\left([0,T];
			L^\infty( {\rr\times} V){\cap L^1(\rr;L^\infty(V))}\right)
		\end{align} 
		to
		\begin{align*}
			\dt f + v\cdot \grd_x f &=\K(f) + h,\\
			f(t=0,x,v) &= \phi(x,v)  \in L^\infty_+(\rr\times V) {\cap L^1(\rr;L^\infty(V))}
		\end{align*}
		that is  bounded 
		\begin{equation*}
			\|f(t)\|_{L^\infty(\rr{\times V})\cap L^1(\rr;L^\infty(V))}\leq e^{2|V|C_Kt}C_\phi + \! \int\limits_0^t \! \! e^{2|V|C_K(t-s)}\|h(s)\|_{L^\infty(\rr\times V) {\cap L^1(\rr;L^\infty(V))}}\rd s.
		\end{equation*}
	\end{lemma}
	We carry out the proof once to make clear, how the constant in the bound is derived.
	\begin{proof}
		Rewrite \eqref{eq:chemotaxis} as 
		\[\dt f = \A f + \B f  + h\]
		with operators $\A:\D(\A)\to \X, f\mapsto - v\cdot \grd_x f$ and $\B:\X\to \X, f\mapsto \K(f)$, where the function spaces  $\D(\A):= W^{1,{p}}(\rr;L^\infty(V))$ and $\X:={L^{p}(\rr;L^\infty(V))}$ are used. The transport operator $\A$ generates a strongly continuous semigroup $T(t)u(x) = u(x-vt)$ { on $\X$} with  operator norm $\|T(t)\|\leq 1$. Clearly, $\B$ is bounded in operator norm by $2|V|C_K$. The bounded perturbation theorem, see e.g. \cite{EngelNagel_2001_Semigroups}, shows that $\A+\B$ generates a strongly continuous semigroup $S$ with $\|S(t)\|\leq e^{2|V|C_Kt}$. {As $\phi \in \X$} {and $h\in L^1((0,T);\X)$ by Hölder's inequality}, \eqref{eq:chemotaxis}  admits a {mild} solution 
		\[
		f(t) = S(t) \phi + \int_0^t S(t-s)h(s)\rd s.
		\]
		{Furthermore, a uniform in $p$ bound on $f$ can be derived with $\|\phi\|_{L^p(\rr;L^\infty(V))}\leq C_\phi$ and $\|h(s)\|_{L^p(\rr;L^\infty(V))}\leq\|h(s)\|_{L^\infty(\rr\times V)\cap L^1(\rr;L^\infty(V))} $ as
			\[
			\|f(t)\|_{L^p(\rr;L^\infty(V))}\leq e^{2|V|C_Kt}C_\phi + \int_0^t e^{2|V|C_K(t-s)}\|h(s)\|_{L^\infty(\rr\times V) {\cap L^1(\rr;L^\infty(V))}}\rd s,
			\]
			which provides the $L^\infty$ bound in the assertion as $p\to \infty$.}
		
	\end{proof}
	{Uniform boundedness of the solution of \eqref{eq:chemotaxis}--\eqref{eq:chemotaxis_init} is a direct consequence of this lemma.}
	\begin{corollary}\label{cor:existencerefularityf}
		Let Assumption  \ref{ass:all} hold. 
		Equation \eqref{eq:chemotaxis} has a {mild} solution $f$ {that} is {uniformly} bounded  {by}
		\begin{equation}\label{eq:boundf}
			\|f(t)\|_{L^\infty(\rr{;L^\infty(V)}){\cap L^1(\rr;L^\infty(V))}}\leq e^{2|V|C_Kt}C_\phi \leq e^{2|V|C_KT}C_\phi=: C_f.
		\end{equation}
	\end{corollary}
	Again, semigroup theory shows existence of the adjoint equation \eqref{eq:adjoint} with corresponding bounds.
	
	\begin{lemma}\label{lem:existenceg}
		Let $h \in {L^1}((0,T);L^\infty(V;L^1( \rr)))$, $\psi \in{L^1}(\rr)$  and let \eqref{eq:Kbounded} hold. 
		Then the equation 
		\begin{align}\label{eq:gGeneraleq}
			-\dt g - v\cdot \nabla_x g &=  \alpha \tilde{\L}(g)  - \sigma g + h,\\
			g(t=T) &= \psi(x)\nonumber
		\end{align}
		with $\alpha\in \{0,1\}$ and $\tilde \L(g) := \int K(x,v',v)g(x,t,v')\rd v'  $  and $\sigma(x,v) := \int K(x,v',v)\rd v'$ 
		has a {mild} solution  
		\begin{align}\label{eq:regularityg}
			g \in C^0\left([0,T];L^\infty(V;{L^1}(\rr))\right)
		\end{align}  
		that satisfies
		\begin{align}\label{eq:gestimate}
			\| g(t)\|_{L^\infty(V;L^1(\rr))}  \leq e^{(1+\alpha)|V|C_K(T-t)}\left(\| \psi\|_{L^1(\rr)} +  \int_0^{T-t} \max_v \|h(T-s,v)\|_{L^1(\rr)} \rd s\right).
		\end{align}
		If, additionally, $h\in L^\infty([0,T]\times V; L^1(\rr))$, then
		\begin{align}\label{eq:gestimate_e-1}
			&\| g(t)\|_{L^\infty(V;L^1(\rr))}  \\
			&\quad \leq e^{(1+\alpha)|V|C_K(T-t)}\| \psi\|_{L^1(\rr)}+  \frac{e^{(1+\alpha)|V|C_K(T-t)}-1}{(1+\alpha)|V|C_K} \ess_{t,v} \|h(t,v)\|_{L^1(\rr)} .\nonumber
		\end{align}
	\end{lemma}
	\begin{proof}
		Repeating the arguments in the proof of Lemma \ref{lem:existencef}, semigroup theory yields the existence of a {mild} solution 
		\begin{align*}
			g(t)  = S(T-t) \psi + \int_0^{T-t}S(T-t-s)h(T-s)\rd s 
		\end{align*}
		for the semigroup $S(t)$ generated by the operator $v\cdot\grd_x + \alpha\tilde{\L} - \sigma$ with $\|S(t)\|\leq e^{(1+\alpha)|V|C_Kt}$.
		This yields \eqref{eq:gestimate} and \eqref{eq:gestimate_e-1}.
	\end{proof}
	
	\section{Proof of Lemma~\ref{lem:IllConditioning:fgn}-\ref{lem:IllCond:Remainder}}\label{appendix:lemma3}
	In this section, we provide the proof for the two Lemmas in section~\ref{ssec:IllCond}. In particular, Lemma~\ref{lem:IllConditioning:fgn} discusses the smallness of the first term in~\eqref{eqn:spell_g_bar}.
	
	\begin{proof}[Proof for Lemma~\ref{lem:IllConditioning:fgn}]
		Let $h\in L^\infty([0,T]\times \rr^d\times V)$. For $n=0$, one has the 
		explicit representation
		\begin{align}\label{eq:g0formula}
			\bar g_{{m}}^{{(0)}}(t,x,v_0) = e^{-\int_0^{T-t} \sigma(x+v_0\tau,v_0) \rd \tau } {(\mu_2- \mu_{{1,m}})(x+v_0(T-t))}.
		\end{align}
		This shows that 
		\begin{align*}
			&\iiint\limits_{[0,T]\times \rr\times V}h\bar g^{(0)}_m\rd(t,x,v)\\ 
			&= \iiint\limits_{[0,T]\times \rr\times V}h(t,x,v)e^{-\int_0^{T-t} \sigma(x+v\tau,v) \rd \tau } {(\mu_2- \mu_{{1,m}})(x+v(T-t))} \rd(t,x,v)\\
			&= \int_\rr \iint\limits_{[0,T]\times V}h(t,x-v(T-t),v)e^{-\int_0^{T-t} \sigma(x-v(T-t-\tau),v) \rd \tau } \rd (t,v) {(\mu_2- \mu_{{1,m}})(x)} \rd x\\
			&\xrightarrow[\text{weakly in } L^1]{ \mu_2-\mu_{1,m}\rightharpoonup 0} 0,
		\end{align*}
		noting that the inner integral lies in $L^\infty( \rr^d)$. 
		For induction, assume that the assertion holds for $\bar g_{n-1}$. 
		Using the explicit form
		\[
		\bar g^{(n)}_m(t,x,v) = \int_t^T e^{-\int_0^{s-t}\sigma(x+v\tau,v)\rd \tau}\int_V K(x+v(s-t), v',v) \bar g^{(n-1)}_m(s,x+v(s-t),v')\rd v'\rd s,
		\]
		  a change of variables in $x$ shows that the  assertion holds for $g^{(n)}_m$ as well, as
		\begin{align*}
			&\iiint\limits_{[0,T]\times \rr\times V}h\bar g^{(n)}_m\rd(t,x,v)\\ 
			&=  \hspace{.cm}\iiint\limits_{[0,T]\times \rr\times V} \bar g^{(n-1)}_m(s,x, v')  \iint\limits_{[0,s]\times V} h(s-t,x-vt,v) e^{-\int_0^{t}\sigma(x-v\tau, v)\rd \tau}K(x,v',v)\rd (t,v)  \rd (s,x, v') \\
			&\xrightarrow[\text{weakly in } L^1]{\bar g^{(n-1)}_m\rightharpoonup 0} 0,
		\end{align*}
		by the fact that the inner integral is in $L^\infty([0,T]\times \rr^d\times V)$ by uniform boundedness of $h$ and $K$.
	\end{proof}
	
	Lemma~\ref{lem:IllCond:Remainder} argues the smallness of the second term in~\eqref{eqn:spell_g_bar}. We provide the proof below. It is a consequence of the smallness of $\bar g_{{m}}^{{(>N)}}$ by Lemma \ref{lem:existenceg} and the boundedness of $f$.
	\begin{proof}[Proof for Lemma~\ref{lem:IllCond:Remainder}]
		Application of lemma \ref{lem:existenceg}  to  $g=\bar g_{{m}}^{{(>N)}}, h = \tilde \L \bar g_{{m}}^{{(n)}}, \alpha = 1$ and $\psi = 0$ yields
		\begin{align*}
			\max_v \int_\rr |\bar g_{{m}}^{{(>N)}}(t)|\rd x &\leq e^{2C_K|V|(T-t)}\int_0^{T-t}\sup_{v}\|\tilde \L (\bar g_{{m}}^{{(n)}})(T-s,v)\|_{L^1(\rr)}\rd s \\
			&\leq |V|C_K(T-t) e^{2C_K|V|(T-t)}  \ess_{s,v}\| \bar g_{{m}}^{{(n)}}(s,x,v)\|_{L^1(\rr)}.
		\end{align*}
		Now, application of the same lemma to the evolution equation \eqref{eq:gn}  for $\bar g_{{m}}^{{(n)}}$, $n=1,...,N$, shows 
		\begin{align*}
			\ess_{t,v} \int_\rr |\bar g_{{m}}^{{(n)}}|\rd x \leq (e^{C_K|V|T}-1)\ess_{s,v}\int_\rr|\bar g_{{m}}^{{(n-1)}}(s,x,v)|\rd x.
		\end{align*}
		The boundedness of $f$ in \eqref{eq:boundf} and repeated application of the above estimate lead to 
		\begin{align*}
			&\left|\int_0^T\max_v \int_\rr f'\bar g_{{m}}^{{(>N)}}\rd x \rd t\right|\\
			&\leq \frac{T^2}{2}|V|C_K C_\phi e^{2|V|C_KT} (e^{C_K|V|T}-1)^N\ess_{s,v}\int_\rr|\bar g_{{m}}^{{(0)}}(s,x,v)|\rd x\\
			& \leq  \frac{T^2}{2} |V|C_K C_\phi e^{2|V|C_KT} \left(e^{C_K|V|T}-1\right)^N\ess_{s,v}\int_\rr|{(\mu_2-\mu_{{1,m}})}
			(x+vs)|\rd x \\
			&\leq  T^2 |V|C_K C_\phi e^{2|V|C_KT} (e^{C_K|V|T}-1)^N C_\mu,
		\end{align*}
		where   $|\bar g_{{m}}^{{(0)}}(s,x,v)| \leq |{(\mu_2-\mu_{{1,m}})}(x+v{(T-s)})|$ can be observed from the explicit formula \eqref{eq:g0formula} for $ \bar g_{{m}}^{{(0)}}$.
	\end{proof}
	
	\section{Proof of Lemmas in Section~\ref{sec:ExpDesign}}\label{ap:DesignWellposed:ProofSsmall}
	We provide proofs for Lemma~\ref{lem:DesignWellPosed:f0g0}-\ref{lem:DesignWellPosed:S} in this section.
	\begin{proof}[Proof of Lemma~\ref{lem:DesignWellPosed:f0g0}]
		Use the explicit representations 
		\begin{align}
			{g_{{1}}^{{(0)}}}(t,x,v) &= e^{-(T-t)\sigma_1(v)}\mu_{{1}}(x+v(T-t)), \label{eq:g10explicit}\\
			f^{{(0)}}{\mid_{I_1}}(t,x,v) &= e^{-t\sigma_1(v)}\phi_{{1}}(x-vt)\label{eq:f0explicit}
		\end{align}
		with $\sigma_1(v) = \int_V K_1(v',v)\rd v'$ and set without loss of generality ${a_{1/2}} = 0$. 
		One obtains for $(v,v') = (+1,-1)$
		\begin{align*}
			& \int_0^T \int_{I_1} f^{{(0)}}(v')({g_{{1}}^{{(0)}}}(v')- {g_{{1}}^{{(0)}}}(v))\rd x \rd t\\
			&=  \int_0^T  \int_{I_1} e^{-t\sigma_1(v')}\phi_{{1}}(x-v't)\big(e^{-(T-t)\sigma_1(v')}\mu_{{1}}(x+v'(T-t))\\
			&\hspace{4.7cm}- e^{-(T-t)\sigma_1(v)}\mu_{{1}}(x+v(T-t))\big)\rd x \rd t  \\
			&\geq  e^{-T\sigma_1(-1)}T\int_{a_0+T}^{a_1} \phi_{{1}}(x)\mu_{{1}}(-T+x)\rd x -\int_{T-\frac{d_\mu+d}{2}}^T\int_{I_1} \phi_{{1}}(x)\mu_{{1}}(-T+x)\rd x \rd t\\
			&\geq e^{-TC_K|V|}TC_{\phi\mu} - \frac{d_\mu+d}{2}C_{\phi\mu},
		\end{align*}
		where the first inequality is due to the fact that  $\phi_{{1}}(x-v't)\mu_{{1}}(x+v(T-t))= \phi_{{1}}(x+t)\mu_{{1}}(x+(T-t))\neq 0$ only for $x\in [-t-d, -t+d]\cap [-2T+t-d_\mu,-2T+t+d_\mu ]\subset I_1$ which is   empty  for $t\leq T-\frac{d_\mu +d}{2}$.\\
		For {the second term}
		, instead, we obtain
		\begin{align*}
			& \left| \int_0^T \int_{I_1} f^{{(0)}}(v)({g_{{1}}^{{(0)}}}(v)- {g_{{1}}^{{(0)}}}(v'))\rd x \rd t\right|\\
			&= \bigg| \int_0^T  \int_{I_1} e^{-t\sigma_1(v)}\phi_{{1}}(x-vt)\big(e^{-(T-t)\sigma_1(v)}\mu_{{1}}(x+v(T-t))\\
			&\hspace{4.7cm}- e^{-(T-t)\sigma_1(v')}\mu_{{1}}(x+v'(T-t))\big)\rd x \rd t\bigg|  \\
			&\leq C_{\phi\mu} \frac{d+d_\mu}{2}
		\end{align*}
		since 
		\begin{itemize}
			\item $\phi_{{1}}(x-vt)\mu_{{1}}(x+v(T-t)) = \phi_{{1}}(x-t)\mu_{{1}}(x+T-t) $ vanishes, as its support $[t-d, t+d]\cap [-2T+t-{d_\mu}{}, -2T+t+{d_\mu}{}]  = \emptyset $ is empty by construction of $T>d\geq d_\mu$ and 
			\item the support $[t-d, t+d]\cap [-t-{d_\mu}{}, -t+{d_\mu}{}] $ of $\phi_{{1}}(x-vt)\mu_{{1}}(x+v'(T-t)) = \phi_{{1}}(x-t)\mu_{{1}}(x-(T-t))$  is non-empty only for $t\leq \frac{d+d_\mu}{2}$.
		\end{itemize}
		Since $e^{-TC_K|V|}- \frac{d_\mu +d}{T} >0$ by assumption, this proves the assertion.
	\end{proof}
	To show  inequality \eqref{eq:DesignWellPosed:Sbound} in Lemma~\ref{lem:DesignWellPosed:S}, {fix}  some $N\in \mathbb{N}$ to be determined later {and decompose}
	\begin{align}\label{eq:DesignWellPosed:Sdecompose}
		S=  &\sum_{\substack{n,k=0\\n+k\geq 1}}^N \int_0^T\int_{I_1} f^{{(k)}}(v')({g_1^{{(n)}}}(v')-{g_1^{{(n)}}}(v))\rd x\rd t \nonumber\\&+  \int_0^T\int_{I_1} f(v')({g_1^{{(>N)}}}(v')-{g_1^{{(>N)}}}(v))\rd x\rd t  \\
		&+ \sum_{n=0}^N \int_0^T\int_{I_1} f^{{(>N)}}(v')({g_1^{{(n)}}}(v')-{g_1^{{(n)}}}(v))\rd x\rd t\, ,\nonumber
	\end{align}
	where ${g_1^{{(n)}}}$ and ${g_1^{{(>N)}}}$  solve \eqref{eq:gn} and \eqref{eq:g>N} respectively and $f^{{(k)}}$ are    solutions   to
	\begin{align*}
		\partial_t f^{{(k)}} - v\cdot \nabla_x f^{{(k)}} & = \L(f^{{(k-1)}})-\sigma f^{{(k)}}, \\
		f^{{(k)}} (t=0,x,v)& = 0,
	\end{align*}
	with $\L(h):= \int_V K(v,v')h(t,x,v')\rd v'$, and $f^{{(>N)}}$ satisfies 
	\begin{align*}
		\partial_t f^{{(>N)}} - v\cdot \nabla_x f^{{(>N)}}& = \L(f^{{(N)}} + f^{{(>N)}})-\sigma f^{{(>N)}}, \\
		f^{{(>N)}} (t=0,x,v)& = 0.
	\end{align*}
	Each part of $S$ in representation \eqref{eq:DesignWellPosed:Sdecompose} is estimated separately in the subsequent three lemmas.
	
	\begin{lemma}\label{lem:DesignWellPosed:fkgn}
		In the setting of proposition \ref{prop:ourDesignWellPosed}, 
		\begin{align*}
			&\left|\int_0^T\int_{I_1} f^{{(k)}}(v')({g_1^{{(n)}}}(v')- {g_1^{{(n)}}}(v))\rd x \rd t \right|\leq 2 \max_{v,v'} \int_0^T\int_{I_1}  f^{{(k)}}(v'){g_1^{{(n)}}}(v)\rd x \rd t\\
			&\leq 2\left(C_K|V|\right)^{n+k}T^{n+k+1}C_{\phi\mu}.
		\end{align*}
	\end{lemma}
	\begin{proof}
		{Iterative application of the source iterations \eqref{eq:gsourceiteration} and
			\begin{align*}
				f^{{(k)}}(t,x,v_0) &= \int_0^t\int_Ve^{-s_0 \sigma(v_0)}K({x-v_0s_0}, v_0,  v_1) f^{{(k-1)}}(t-s_0,x-v_0s_0, v_1)\rd  v_1\rd s_0,
			\end{align*}
			together with the bound $e^{-s\sigma(v)} K(x,v',v'') \leq C_K$ and the explicit formulas \eqref{eq:g10explicit}--\eqref{eq:f0explicit} provides the estimates
			\begin{align}
				0\leq {g_1^{{(n)}}}(x,t,v_0) \leq &\ C_K^n\int \limits_0^{T-t}\int_V...\int\limits_0^{T-t-\sum_{i=0}^{n-2}s_i} \int_V\mu_{{1}}\left(x+\sum_{i=0}^{n-1} v_is_i + v_n\left(T-t-\sum_{i=0}^{n-1} s_i\right)\right)\label{eq:estimategn}\\
				&\hspace{7cm}\rd v_n\rd s_{n-1}...\rd v_1\rd s_0 \nonumber,\\
				0\leq f^{{(k)}}(x,t,v_0) \leq  &\ C_K^k\int\limits_0^{t}\int_V...\int\limits_0^{ t-\sum_{i=0}^{k-2}s_i}\int_V \phi\left(x-\sum_{i=0}^{k-1} v_is_i + v_k\left(t-\sum_{i=0}^{k-1} s_i\right)\right)\nonumber\\
				&\hspace{7cm}\rd {v}_k\rd s_{k-1}...\rd v_1\rd s_0.\nonumber
			\end{align}
		}
		{Using again $f^{{(k)}}|_{I_1} =f_{{1}}^{{(k)}}$ with initial condition $\phi_{{1}}$ in the notation of the proof of Proposition \ref{prop:Designdecouple}, t}his proves
		\begin{align*}
			&\left|\int_0^T\int_{I_1} f^{{(k)}}(v')({g_1^{{(n)}}}(v')- {g_1^{{(n)}}}(v))\rd x \rd t \right|\leq 2 \max_{v,v'}\int_0^T\int_{I_1} f_{{1}}^{{(k)}}(v'){g_1^{{(n)}}}(v)\rd x \rd t\\
			&\leq 2\left({C_K}|V|\right)^{n+k}T^{n+k+1}C_{\phi\mu}.
		\end{align*}
	\end{proof}
	The following bound for the second summand in \eqref{eq:DesignWellPosed:Sdecompose} { follows directly from }
	Lemma \ref{lem:IllCond:Remainder}.
	\begin{lemma}\label{lem:DesignWellPosed:fg>N}
		In the setting of Proposition \ref{prop:ourDesignWellPosed}, 
		\begin{align*}
			&\max_{v}  \left|\iint f(v')({g_1^{{(>N)}}}(v')-{g_1^{{(>N)}}}(v))\rd x\rd t \right|\\
			&\leq 4 T^2|V|C_KC_\phi  e^{2|V|C_KT}(e^{C_K|V|T}-1)^N \bar C_\mu  d_\mu=:C'(T) (e^{C_K|V|T}-1)^N{.}
		\end{align*}
	\end{lemma}
	For the third term in \eqref{eq:DesignWellPosed:Sdecompose}, one establishes the following bound.
	\begin{lemma}\label{lem:DesignWellPosed:f>Ngn}
		In the setting of Proposition \ref{prop:ourDesignWellPosed},
		\begin{align*}
			&\max_v\left|\iint f^{{(>N)}}(v')({g_1^{{(n)}}}(v')- {g_1^{{(n)}}}(v))\rd x \rd t \right| \\
			&\leq 4 T^2|V|C_KC_\phi e^{2|V|C_KT}(e^{C_K|V|T}-1)^N \bar C_\mu d_\mu \left({C_K|V|T}\right)^n  \\
			&= {C'}(T)(e^{C_K|V|T}-1)^N \left({C_K|V|T}\right)^n{.}
		\end{align*}
	\end{lemma}
	\begin{proof}
		An estimate for $f^{{(>N)}}$ can be derived analogously as the estimate for ${\bar g_{{1}}^{{(n)}}}$ in Lemma \ref{lem:IllCond:Remainder} from Lemma \ref{lem:existencef}
		\begin{align*}
			\|f^{{(>N)}}\|_{L^\infty([0,T]\times\rr\times V)}\leq |V|C_KTe^{2|V|C_KT}(e^{C_K|V|T}-1)^NC_\phi.
		\end{align*}
		Together with \eqref{eq:estimategn}, this proves the lemma.
	\end{proof}
	Lemma \ref{lem:DesignWellPosed:S} can now be  assembled from the previous lemmas.
	\begin{proof}[Proof of Lemma \ref{lem:DesignWellPosed:S}]
		Lemmas \ref{lem:DesignWellPosed:fkgn}, \ref{lem:DesignWellPosed:fg>N} and \ref{lem:DesignWellPosed:f>Ngn} yield the $(v,v')$ independent bound 
		\begin{align*}
			|S|&\leq 2C_{\phi\mu}T\sum_{\substack{n,k=0\\n+k\geq 1}}^N \left({C_K |V|T}\right)^{n+k} + (e^{C_K|V|T}-1)^N {C'(T)}\left(1+  \sum_{n=0}^N\left({C_K|V|T}\right)^n\right)\\
			&\leq 4 C_{\phi\mu}T\frac{C_K|V|T}{(1-C_K|V|T)^2}
			+ (e^{C_K|V|T}-1)^N {C'(T)}\left(1+ \frac{1}{1-C_K|V|T}\right)\\
			&=: 4 C_{\phi\mu}T\frac{C_K|V|T}{(1-C_K|V|T)^2}
			+ (e^{C_K|V|T}-1)^N C(T).
		\end{align*}
		Because $e^{C_K|V|T}-1<1$ due to the  assumption $T<(1-\delta)\frac{0.09}{C_K|V|}$, the second term in the last line becomes arbitrarily small for  large $N\in \mathbb{N}$, which shows that $|S|$ is in fact bounded by the first term.
	\end{proof}

\subsection*{Funding}
K.H. acknowledges support by the  {German Academic Scholarship Foundation} (Studienstiftung des deutschen Volkes) and the {Marianne-Plehn-Program}..\\
Q.L. is partially supported by Vice Chancellor for Research and Graduate Education, DMS-2308440 and ONR-N000142112140.\\
M.T. is partially supported by the Strategic Priority Research Program of Chinese Academy of Sciences, XDA25010401 and NSFC12031013.

\bibliographystyle{abbrv}
\bibliography{lit.bib}

\end{document}